\documentclass[a4paper]{amsart}

\usepackage{amssymb, enumerate}
\usepackage{aliascnt}
\usepackage{hyperref}
\usepackage{tikz}
\usepackage{xy} \xyoption{all}

\newtheorem{lma}{Lemma}[section]

\newaliascnt{thmCt}{lma}
\newtheorem{thm}[thmCt]{Theorem}
\aliascntresetthe{thmCt}

\newaliascnt{corCt}{lma}
\newtheorem{cor}[corCt]{Corollary}
\aliascntresetthe{corCt}

\newaliascnt{prpCt}{lma}
\newtheorem{prp}[prpCt]{Proposition}
\aliascntresetthe{prpCt}

\theoremstyle{definition}

\newaliascnt{pgrCt}{lma}
\newtheorem{pgr}[pgrCt]{}
\aliascntresetthe{pgrCt}

\newaliascnt{dfnCt}{lma}
\newtheorem{dfn}[dfnCt]{Definition}
\aliascntresetthe{dfnCt}

\newaliascnt{rmkCt}{lma}
\newtheorem{rmk}[rmkCt]{Remark}
\aliascntresetthe{rmkCt}

\newaliascnt{ntnCt}{lma}
\newtheorem{ntn}[ntnCt]{Notation}
\aliascntresetthe{ntnCt}

\newcounter{theoremintro}

\newtheorem{thmIntro}[theoremintro]{Theorem}
\newtheorem{prpIntro}[theoremintro]{Proposition}

\newaliascnt{corIntroCt}{theoremintro}
\newtheorem{corIntro}[corIntroCt]{Corollary}
\aliascntresetthe{corIntroCt}

\newcommand{\NN}{\mathbb{N}}
\newcommand{\calU}{{\mathcal{U}}}
\newcommand{\PU}{\mathcal{PU}}
\newcommand{\sa}{\mathrm{sa}}
\newcommand{\vect}[1]{\mathbf{#1}}
\newcommand{\ca}{$C^*$-algebra}
\newcommand{\grPre}{\gr_0}
\newcommand{\freeVar}{\_\,}

\DeclareMathOperator{\diag}{diag}
\DeclareMathOperator{\Gen}{Gen}
\DeclareMathOperator{\rr}{rr}
\DeclareMathOperator{\gr}{gr}

\DeclareMathOperator{\Prim}{Prim}
\DeclareMathOperator{\Aut}{Aut}
\DeclareMathOperator{\hull}{hull}
\DeclareMathOperator{\locdim}{locdim}
\DeclareMathOperator{\Sub}{Sub}
\DeclareMathOperator{\stab}{Stab}
\DeclareMathOperator{\depth}{depth}
\DeclareMathOperator{\topdim}{topdim}
\newcommand{\SubSep}{\Sub_{\mathrm{sep}}}

\title{The generator rank of subhomogeneous $C^*$-algebras}

\author{Hannes Thiel}
\address{Hannes Thiel,
Department of Mathematics, Kiel University, Heinrich-Hecht-Platz~6, 24118 Kiel, Germany.}
\email{hannes.thiel@math.uni-kiel.de}
\urladdr{www.hannesthiel.org}

\thanks{The author was partially supported by the Deutsche Forschungsgemeinschaft (DFG, German Research Foundation) under Germany's Excellence Strategy EXC 2044-390685587 (Mathematics M\"{u}nster: Dynamics-Geometry-Structure).
}
\subjclass[2010]%
{Primary
46L05, 
46L85; 
Secondary
46M20, 
54F45, 
55M10. 
}
\date{\today}

\begin{document}

\begin{abstract}
We compute the generator rank of a subhomgeneous \ca{} in terms of the covering dimension of the pieces of its primitive ideal space corresponding to irreducible representations of a fixed dimension.
We deduce that every $\mathcal{Z}$-stable ASH-algebra has generator rank one, which means that a generic element in such an algebra is a generator.

This leads to a strong solution of the generator problem for classifiable, simple, nuclear \ca{s}:
a generic element in each such algebra is a generator.
Examples of Villadsen show that this is not the case for all separable, simple, nuclear \ca{s}.
\end{abstract}

\maketitle

\section{Introduction}

The generator rank of a unital, separable \ca{} $A$ is the smallest integer $n\geq 0$ such that the self-adjoint $(n+1)$-tuples that generate $A$ as a \ca{} are dense in $A^{n+1}_\sa$;
see \autoref{dfn:gr} for the nonunital and nonseparable case.
This invariant was introduced in \cite{Thi21GenRnk} to study the generator problem, which asks to determine the minimal number of (self-adjoint) generators for a given \ca.

One difficulty when studying the generator problem is that the minimal number of generators for a \ca{} can increase when passing to ideals or inductive limits.
The main advantage of the generator rank is that it enjoys nice permanence properties:
it does not increase when passing to ideals, quotients or inductive limits;
see \autoref{sec:gr}.

For example, using these permanence properties, one can easily show that AF-algebas have generator rank at most one.
In particular, every AF-algebra is generated by two self-adjoint elements, which solves the generator problem for this class of algebras;
see \cite[Theorem~7.3]{Thi21GenRnk}.

In this paper, we compute the generator rank of subhomogeneous \ca{s}.
Recall that a \ca{} is said to be $d$-homogeneous ($d$-subhomogeneous) if each of its irreducible representations has dimension (at most) $d$.
The typical example of a $d$-homogeneous \ca{} is $C_0(X,M_d)$ for a locally compact, Hausdorff space~$X$.
Further, a \ca{} is subhomogeneous if and only if it is a sub-\ca{} of $C_0(X,M_d)$ for some $X$ and some~$d$;
see for example \cite[Proposition~IV.1.4.3]{Bla06OpAlgs}.

Subhomogeneous \ca{s} and their inductive limits (called approximately subhomogeneous algebras, or simply ASH-algebras) play an important role in the structure and classification theory of \ca{s} since the algebras covered by the Elliott program are either purely infinite or approximately subhomogeneous.
To be precise, let us say that a \ca{} is \emph{classifiable} if it is unital, separable, simple, nuclear, $\mathcal{Z}$-stable (that is, it tensorially absorbs the Jiang-Su algebra $\mathcal{Z}$) and satisfies the Universal Coefficient Theorem (UCT).
By the recent breakthrough in the Elliott classification program, \cite{GonLinNiu15arX:mainClassification, EllGonLinNiu15arX:classFinDR2, TikWhiWin17QDNuclear}, two classifiable \ca{s} are isomorphic if and only if their Elliott invariants ($K$-theoretic and tracial data) are isomorphic.

Classifiable \ca{s} come in two flavors: stably finite and purely infinite.
Every stably finite, classifiable \ca{} is automatically an ASH-algebra.
A major application of our results is that every $\mathcal{Z}$-stable ASH-algebra has generator rank one;
see Corollary~\ref{thmD} below.
In  \cite{Thi20arX:grZstableRR0}, we show that every $\mathcal{Z}$-stable \ca{} of real rank zero has generator rank one.
This includes all purely infinite, classifiable \ca{s}.
It follows that every classifiable \ca{} has generator rank one and therefore contains a dense $G_\delta$-subset of generators;
see \autoref{corE} below. 

One important aspect of the generator problem is to determine if every separable, simple \ca{} is generated by a single operator (equivalently, by two self-adjoint elements).
While this remains unclear, we can refute the possibility that every separable, simple \ca{} contains a dense set of generators:
Villadsen constructed examples of separable, simple, approximately homogeneous \ca{s} (AH-algebras) of arbitrarily high real rank;
see \cite{Vil99SRSimpleCa}.
Let $A$ be such an AH-algebra with $\rr(A)=\infty$.
By \cite[Proposition~3.10]{Thi21GenRnk} (see \autoref{prp:grPre_gr}), the real rank is dominated by the generator rank, whence $\gr(A)=\infty$.
In particular, for every $n$ the generating self-adjoint $n$-tuples (if there are any) are not dense in~$A^{n}_\sa$.

In \cite[Theorem~3.8]{ThiWin14GenZStableCa}, the author and Winter showed that every unital, separable, $\mathcal{Z}$-stable \ca{} is singly generated.
The results of this paper and of \cite{Thi20arX:grZstableRR0} show that under additional assumptions, a (unital) separable, $\mathcal{Z}$-stable \ca{} even contains a dense set of generators.
This raises the natural question if every $\mathcal{Z}$-stable \ca{} has generator rank one;
see \cite[Remarks~5.8(2)]{Thi20arX:grZstableRR0}.

Given a locally compact, Hausdorff space $X$, the local dimension $\locdim(X)$ is defined as the supremum of the covering dimension of all compact subsets, with the convention that $\locdim(\emptyset)=-1$.
For $\sigma$-compact (in particular, second countable) locally compact, Hausdorff spaces, the local dimension agrees with the usual covering dimension (in general they differ).
In \autoref{sec:hom}, we compute the generator rank of arbitrary homogeneous \ca{s}:

\begin{thmIntro}[{\ref{prp:gr_hom}}]
\label{thmA}
Let $A$ be a $d$-homogeneous \ca{}.
Set $X:=\Prim(A)$.
If $d=1$, then $\gr(A) = \locdim(X\times X)$.
If $d\geq 2$, then:
\begin{align*}
\gr(A)
&= \left\lceil\frac{\locdim(X)+1}{2d-2}\right\rceil.
\end{align*}
\end{thmIntro}

In particular, $\gr(C(X,M_d)) = \left\lceil\tfrac{\dim(X)+1}{2d-2}\right\rceil$ if $X$ is a compact, Hausdorff space and $d\geq 2$.
To prove Theorem~\ref{thmA}, we first show a Stone-Weierstra{\ss}-type result that characterizes when a tuple generates $C(X,M_d)$:
the tuple has to generate $M_d$ pointwise, and it has to suitably separate the points in $X$;
see \autoref{prp:generators_C(X)-alg}.
This indicates the general strategy to determine when generating $n$-tuples in $C(X,M_d)$ are dense:
first, we need to characterize when every tuple can be approximated by tuples that generate $M_d$ pointwise;
and second we need to characterize when a pointwise generating tuple can be approximated by tuples that separate the points.
To address the first point, we compute the codimension of the manifold of generating $n$-tuples of self-adjoint $d$-matrices;
see \autoref{prp:dimEdm}.
For the second point, we use known results characterizing when continuous maps to a manifold can be approximated by embeddings, in conjunction with a suitable version of the homotopy extension lifting property.

In \autoref{sec:subhom}, we compute the generator rank of $d$-subhomogeneous \ca{} by induction over $d$.
Given a $d$-subhomogeneous \ca{} $A$, we consider the ideal $I\subseteq A$ corresponding to irreducible representations of dimension $d$.
Then $A/I$ is $(d-1)$-subhomogeneous.
Using Theorem~\ref{thmA} and the assumption of the induction, we know the generator rank of $I$ and $A/I$.
The crucial result to compute the generator rank of the extension is following proposition, which we also expect to have further applications in the future.

\begin{prpIntro}[\ref{prp:CST-gr}]
Let $A$ be a separable \ca{}, and let $(I_k)_{k\in\NN}$ be a \emph{decreasing} sequence of ideals satisfying $\bigcup_k \hull(I_k)=\Prim(A)$.
Then
\[
\gr(A) = \sup_k \gr(A/I_k).
\]
\end{prpIntro}

The main result of this paper is:

\begin{thmIntro}[\ref{prp:gr_subhom}]
\label{thmD}
Let $A$ be a subhomogeneous \ca{}.
For each $d\geq 1$, set $X_d:=\Prim_d(A)$, the subset of the primitive ideal space of $A$ corresponding to $d$-dimensional irreducible representations.
Then:
\begin{align*}
\gr(A)
&= \max\left\{ \locdim(X_1\times X_1), \max_{d\geq 2} \left\lceil\frac{\locdim(X_d)+1}{2d-2}\right\rceil \right\}.
\end{align*}
\end{thmIntro}

The main application is:

\begin{corIntro}[\ref{prp:gr_ZstableASH}]
\label{corD}
Let $A$ be a nonzero, separable, $\mathcal{Z}$-stable ASH-algebra.
Then $\gr(A)=1$, and so a generic element of $A$ is a generator.
\end{corIntro}

It was shown in \cite[Theorem~3.8]{ThiWin14GenZStableCa} that every \emph{unital}, separable, $\mathcal{Z}$-stable \ca{} is singly generated.
We note that \autoref{corD} does not require unitality.
In particular, \autoref{corD} implies that certain \ca{s} are singly generated that were not considered in \cite{ThiWin14GenZStableCa}.

Together with the main result of \cite{Thi20arX:grZstableRR0}, we obtain the following consequence:

\begin{corIntro}[{\cite[Corollary~5.7]{Thi20arX:grZstableRR0}}]
\label{corE}
Let $A$ be a unital, separable, simple, nuclear, $\mathcal{Z}$-stable \ca{} satisfying the UCT. 
Then $A$ has generator rank one. 
In particular, a generic element in $A$ is a generator.
\end{corIntro}

\subsection*{Acknowledgments}

I want to thank the anonymous referee for his thorough reading of a former version of this work and for providing detailed feedback that helped greatly to improve the paper.   

\subsection*{Notation}

We set $\NN:=\{0,1,2\ldots\}$.
Given a \ca{} $A$, we use $A_\sa$ to denote the set of self-adjoint elements in $A$.
We denote by $\widetilde{A}$ the minimal unitization of $A$.
By an ideal in a \ca{} we mean a closed, two-sided ideal.
We write $M_d$ for the \ca{} of $d$-by-$d$ matrices $M_d(\mathbb{C})$.

Given $a,b\in A$, and $\varepsilon>0$, we write $a=_\varepsilon b$ if $\|a-b\|<\varepsilon$.
Given $a\in A$ and $G\subseteq A$, we write $a\in_\varepsilon G$ if there exists $b\in G$ with $a=_\varepsilon b$.
We use bold letters to denote tuples, for example $\vect{a}=(a_1,\ldots,a_n)\in A^n$.
Given $\vect{a},\vect{b}\in A^n$, we write $\vect{a}=_\varepsilon\vect{b}$ if $a_j=_\varepsilon b_j$ for $j=1,\ldots,n$.
We use $C^*(\vect{a})$ to denote the sub-\ca{} of $A$ generated by the elements of $\vect{a}$.
We write $A^n_\sa$ for $(A_\sa)^n$, the space of $n$-tuples of self-adjoint elements in $A$.

\section{The generator rank and its precursor}
\label{sec:gr}

In this section, we briefly recall the definition and basic properties of the generator rank $\gr$ and its predecessor $\grPre$ from \cite{Thi21GenRnk}.

\begin{dfn}[{\cite[Definitions~2.1, 3.1]{Thi21GenRnk}}]
\label{dfn:gr}
Let $A$ be a \ca{}.
We define $\grPre(A)$ as the smallest integer $n\geq 0$ such that for every $\vect{a}\in A^{n+1}_\sa$, $\varepsilon>0$ and $c\in A$, there exists $\vect{b}\in A^{n+1}_\sa$ such that
\[
\vect{b}=_\varepsilon\vect{a}, \quad\text{ and }\quad
c\in_\varepsilon C^*(\vect{b}).
\]
If no such $n$ exists, we set $\grPre(A)=\infty$.
The \emph{generator rank} of $A$ is $\gr(A):=\grPre(\widetilde{A})$.
\end{dfn}

We use $\Gen_{n}(A)_\sa$ to denote the set of tuples $\vect{a}\in A^n_\sa$ that generate $A$ as a \ca.
For separable \ca{s}, the generator rank and its predecessor can be described by the denseness of such tuples.

\begin{thm}[{\cite[Theorem~3.4]{Thi21GenRnk}}]
\label{prp:grSep}
Let $A$ be a \emph{separable} \ca{} and $n\in\NN$.
Then:
\begin{enumerate}
\item
$\grPre(A)\leq n$ if and only if $\Gen_{n+1}(A)_\sa\subseteq A^{n+1}_\sa$ is a dense $G_\delta$-subset.
\item
$\gr(A)\leq n$ if and only if $\Gen_{n+1}(\widetilde{A})_\sa\subseteq \widetilde{A}^{n+1}_\sa$ is a dense $G_\delta$-subset.
\end{enumerate} 
\end{thm}

\begin{rmk}
\label{rmk:gr1}
Let $A$ be a separable \ca.
If $A$ has generator rank at most one, then the set of (non-self-adjoint) generators in $A$ is a dense $G_\delta$-subset;
see \cite[Remark~3.7]{Thi21GenRnk}.
If $A$ is unital, then the converse also holds:
We have $\gr(A)\leq 1$ if and only if a generic element in $A$ is a generator.
\end{rmk}

The connection between $\grPre, \gr$ and the real rank is summarized by the next result, which combines Proposition~3.12 and Theorem~3.13 in \cite{Thi21GenRnk}.
In \autoref{prp:gr_subhom}, we show that $\grPre$ and $\gr$ agree for subhomogeneous \ca{s}. 
In general, however, it is unclear if $\grPre=\gr$;
see \cite[Question~3.16]{Thi21GenRnk}.

\begin{prp}
\label{prp:grPre_gr}
Let $A$ be a \ca{}.
Then
\[
\max \big\{ \rr(A), \grPre(A) \big\}
= \gr(A)
\leq \grPre(A)+1.
\]
\end{prp}

We will frequently use the following permanence properties of $\grPre$ and $\gr$, which were shown in Propositions~2.2, 2.7, 2.9 and Theorem~6.2 in \cite{Thi21GenRnk}.

\begin{thm}
\label{prp:gr_idealQuotExt}
Let $A$ be a \ca{}, and let $I\subseteq A$ be an ideal.
Then:
\[
\max\big\{ \grPre(I),\grPre(A/I) \big\} \leq \grPre(A) \leq \grPre(I)+\grPre(A/I)+1,
\]
and
\[
\max\big\{ \gr(I),\gr(A/I) \big\} \leq \gr(A) \leq \gr(I)+\gr(A/I)+1.
\]
\end{thm}

Recall that a \ca{} $A$ is said to be \emph{approximated} by sub-\ca{s} $A_\lambda\subseteq A$ if for every finite subset $F\subseteq A$ and $\varepsilon>0$ there is $\lambda$ such that $a\in_\varepsilon A_\lambda$ for each $a\in F$.
We do not require the subalgberas to be nested.
Thus, while $\bigcup_\lambda A_\lambda$ is a dense subset of $A$, it is not necessarily a subalgbra.
The next result combines Propositions~2.3, 2.4, and Theorem~6.3 in \cite{Thi21GenRnk}.

\begin{thm}
\label{prp:gr_approx}
Let $A$ be a \ca{} that is approximated by sub-\ca{s} $A_\lambda\subseteq A$ and let $n\in\NN$.
If $\grPre(A_\lambda)\leq n$ for each $\lambda$, then $\grPre(A)\leq n$.
Analogously, if $\gr(A_\lambda)\leq n$ for each $\lambda$, then $\gr(A)\leq n$.

Moreover, if $A=\varinjlim_{j} A_j$ is an inductive limit, then
\[
\grPre(A)\leq\liminf_{j} \grPre(A_j), \quad\text{ and }\quad
\gr(A)\leq\liminf_{j} \gr(A_j).
\]
\end{thm}

\begin{thm}[{\cite[Theorem~5.6]{Thi21GenRnk}}]
\label{prp:gr_abln}
Let $X$ be a locally compact, Hausdorff space.
Then
\[
\grPre(C_0(X))
= \gr(C_0(X))
= \locdim(X\times X).
\]
\end{thm}

\section{Reduction to the separable case}
\label{sec:reduction}

Let us recall a few concepts from model theory that allow us to reduce some proofs in the following sections to the case of separable \ca{s}.
We refer to \cite{FarKat10NonsepUHF1} and \cite{FarHarLupRobTikVigWin21ModelThy} for details.

\begin{pgr}
\label{pgr:LS}
Let $A$ be a \ca{}.
We use $\SubSep(A)$ to denote the set of separable sub-\ca{s} of $A$.
A collection $\mathcal{S}\subseteq\SubSep(A)$ is said to be \emph{$\sigma$-complete} if we have $\overline{\bigcup\{B:B\in\mathcal{T}\}}\in\mathcal{S}$ for every countable, directed subcollection $\mathcal{T}\subseteq\mathcal{S}$.
Further, $\mathcal{S}$ is said to be \emph{cofinal} if for every $B_0\in\SubSep(A)$ there is $B\in\mathcal{S}$ such that $B_0\subseteq B$.
It is well known that the intersection of countably many $\sigma$-complete, cofinal collections is again $\sigma$-complete and cofinal.

In \cite[Definition~1]{Thi13TopDimTypeI}, I formalized the notion of a \emph{noncommutative dimension theory} as an assignment that to each \ca{} $A$ associates a number $d(A)\in\{0,1,2,\ldots,\infty\}$ such that six axioms are satisfied.
Axioms~(D1)-(D4) describe compatibility with passing to ideals, quotients, directs sums, and unitizations.
The other axioms are:
\begin{enumerate}
\item[(D5)]
If $n\in\NN$ and if $A$ is a \ca{} that is approximated by sub-\ca{s} $A_\lambda\subseteq A$ (as in \autoref{prp:gr_approx}) such that $d(A_\lambda)\leq n$ for each $\lambda$, then $d(A)\leq n$.
\item[(D6)]
If $A$ is a \ca{} and $B_0\subseteq A$ is a separable sub-\ca{}, then there is a separable sub-\ca{} $B\subseteq A$ such that $B_0\subseteq B$ and $d(B)\leq d(A)$.
\end{enumerate}

It was noted in \cite[4.1]{Thi21GenRnk} that if $d$ is an assignment from \ca{s} to $\{0,1,\ldots,\infty\}$ that satisfies~(D5) and~(D6), then for each $n\in\NN$ and each \ca{}~$A$ satisfying $d(A)\leq n$, the collection
\[
\big\{ B \in \SubSep(A) : d(B)\leq n \big\}
\]
is $\sigma$-complete and cofinal.
It was shown in \cite{Thi21GenRnk} that $\grPre$ and $\gr$ satisfy~(D5) and~(D6).
\end{pgr}

\begin{lma}
\label{prp:subcollections}
Let $A$ be a \ca{} and let $I\subseteq A$ be an ideal.
We have:
\begin{enumerate}
\item
Let $\mathcal{S}\subseteq\SubSep(I)$ be a $\sigma$-complete and cofinal subcollection.
Then the family $\{B\in\SubSep(A):B\cap I\in\mathcal{S}\}$ is $\sigma$-complete and cofinal.
\item
Let $\mathcal{S}\subseteq\SubSep(A/I)$ be a $\sigma$-complete and cofinal subcollection.
Then the family $\{B\in\SubSep(A):B/(B\cap I)\in\mathcal{S}\}$ is $\sigma$-complete and cofinal.
\end{enumerate}
\end{lma}
\begin{proof}
(1):
Set $\mathcal{T} := \{ B \in \SubSep(A) : B\cap I \in\mathcal{S} \}$.
It is easy to see that $\mathcal{T}$ is $\sigma$-complete.
To show that it is cofinal, let $B_0\in\SubSep(A)$.
We will inductively find increasing sequences $(I_k)_k$ in $\mathcal{S}$ and $(B_k)_k$ in $\SubSep(A)$ such that
\[
B_0\cap I \subseteq I_0 \subseteq B_1\cap I \subseteq I_1 \subseteq \ldots
\]
Assume that we have obtained $B_k$ for some $k\in\NN$.
Then $B_k\cap I\in\SubSep(I)$, and since $\mathcal{S}$ is cofinal in $\SubSep(I)$, we obtain $I_k\in\mathcal{S}$ such that $B_k\cap I\subseteq I_k$.
Then let $B_{k+1}$ be the sub-\ca{} of $A$ generated by $B_k$ and $I_k$.

Set $B:=\overline{\bigcup_k B_k}$, which belongs to $\SubSep(A)$ and contains $B_0$.
We have $B\cap I = \overline{\bigcup_k I_k}$.
Since $\mathcal{S}$ is $\sigma$-complete, $B\cap I$ belongs to $\mathcal{S}$.
Thus, $B$ belongs to $\mathcal{T}$, as desired.

Statement~(2) is shown similarly.
\end{proof}

\begin{pgr}
\label{pgr:hom}
Recall that a \ca{} is called \emph{$d$-homogeneous} (for some $d\geq 1$) if all its irreducible representations are $d$-dimensional;
and it is called \emph{homogeneous} if it is $d$-homogeneous for some $d$;
see \cite[Definition~IV.1.4.1, p.330]{Bla06OpAlgs}.

Let $A$ be a $d$-homogeneous \ca{}, and set $X:=\Prim(A)$, the primitive ideal space of $A$.
Then $X$ is a locally compact, Hausdorff space and there exists a locally trivial bundle over $X$ with fiber $M_d$ such that $A$ is canonically isomorphic to the algebra of continuous cross-sections vanishing at infinity, with pointwise operations;
see \cite[Theorem~3.2]{Fel61StructureOpFields}.

It follows that the center of $A$ is canonically isomorphic to $C_0(X)$, and this gives $A$ the structure of a continuous $C_0(X)$-algebra, with each fiber isomorphic to $M_d$.
For the definition and results of $C_{0}(X)$-algebras, we refer the reader to \S{2} of \cite{Dad09CtsFieldsOverFD}.
Given a $C_0(X)$-algebra $A$ and a closed subset $Y\subseteq X$, we let $A(Y)$ denote the quotient of $A$ corresponding to~$Y$.
The fiber of $A$ at $x\in X$ is $A(x):=A(\{x\})$.
Given $a\in A$ and $x\in X$, we write $a(x)$ for the image of $a$ in the quotient $A(x)$.
Given $\vect{a}=(a_0,\ldots,a_n)\in A^{n+1}$, we set $\vect{a}(x):=(a_0(x),\ldots,a_n(x))\in A(x)^{n+1}$.

Given a locally compact, Hausdorff space $X$, the \emph{local dimension} of $X$ is
\[
\locdim(X) := \sup \big\{ \dim(K) : K\subseteq X \text{ compact} \big\},
\]
with the convention that $\locdim(\emptyset)=-1$.
As noted in \cite[Paragraph~5.5]{Thi21GenRnk}, if~$X$ is nonempty, then $\locdim(X)$ agrees with the dimension of the one-point compactification of $X$.
If $X$ is $\sigma$-compact, then $\dim(X)=\locdim(X)$.
\end{pgr}

\begin{lma}
\label{prp:LShomUntwisted}
Let $d\geq 1$, $l\in\NN$, and let $X$ be a compact, Hausdorff space satisfying $\dim(X)\leq l$.
Set $A:=C(X,M_d)$.
Then
\[
\mathcal{S} := \big\{ B\in\SubSep(A) : B \text{ $d$-homogeneous}, \locdim(\Prim(B))\leq l \big\}
\]
is $\sigma$-complete and cofinal.
\end{lma}
\begin{proof}
\emph{$\sigma$-completeness:}
Let $\mathcal{T}\subseteq\mathcal{S}$ be a countable, directed family, and set $C:=\overline{\bigcup\{B:B\in\mathcal{T}\}}$.
To show that $C$ is $d$-homogeneous, let $\varrho$ be an irreducible representation of~$C$.
Since $C$ is $d$-subhomogeneous (as a subalgebra of $A$), the dimension of $\varrho$ is at most~$d$.
If $\dim(\varrho)<d$, then the restriction of $\varrho$ to each $B\in\mathcal{T}$ is zero, whence $\varrho=0$, a contradiction.

In \cite[Section~2.2]{BroPed09Limits}, Brown and Pedersen introduce the \emph{topological dimension} of type~$\mathrm{I}$ \ca{s}.
Given a homogeneous \ca{} $D$, the topological dimension $\topdim(D)$ is equal to $\locdim(\Prim(D))$. 
Hence, each $B\in\mathcal{T}$ satisfies $\topdim(B)=\locdim(\Prim(B))\leq l$.
By \cite[Lemma~3]{Thi13TopDimTypeI}, a continuous trace \ca{} (in particular a homogeneous \ca{}) has topological dimension at most $l$ whenever it is approximated by sub-\ca{s} with topological dimension at most $l$.
Hence,
\[
\locdim(\Prim(C))
= \topdim(C) \leq l,
\]
which verifies that $C$ belongs to $\mathcal{S}$, as desired.

\emph{Cofinality:}
Let $B_0\subseteq A$ be a separable sub-\ca.
We identify $A$ with $C(X)\otimes M_d$.
Let $e_{jk}\in M_d$, $j,k=1,\ldots,d$ be matrix units.
Let $C(Y)\subseteq C(X)$ be a separable, unital sub-\ca{} such that $f\in C(X)$ belongs to $C(Y)$ whenever $f\otimes e_{jk}\in B_0$ for some $j,k$.
Using that the real rank satisfies~(D6), let $C(Z)\subseteq C(X)$ be a separable sub-\ca{} containing $C(Y)$ and such that $\rr(C(Z))\leq\rr(C(X))$.
Then
\[
\dim(Z)=\rr(C(Z))\leq\rr(C(X))=\dim(X)\leq l,
\]
and it follows that $C(Z)\otimes M_d\subseteq C(X)\otimes M_d$ has the desired properties.
\end{proof}

\begin{prp}
\label{prp:LShom}
Let $d\geq 1$, $l\in\NN$, and let $A$ be a $d$-homogeneous \ca{} satisfying $\locdim(\Prim(A))\leq l$.
Then
\[
\mathcal{S} := \big\{ B\in\SubSep(A) : B \text{ $d$-homogeneous}, \locdim(\Prim(B))\leq l \big\}
\]
is $\sigma$-complete and cofinal.
\end{prp}
\begin{proof}
As in the proof of \autoref{prp:LShomUntwisted} we obtain that $\mathcal{S}$ is $\sigma$-complete.

\emph{Cofinality:}
Let $B_0\subseteq A$ be a separable sub-\ca.
Let $I\subseteq A$ be the ideal generated by $B_0$.
Then $I$ is $d$-homogeneous and $X:=\Prim(I)$ is $\sigma$-compact.
We view $I$ as a $C_0(X)$-algebra with all fibers isomorphic to $M_d$.
Since the $M_d$-bundle associated to $I$ is locally trivial, and since $X$ is $\sigma$-compact, we can choose a sequence of compact subsets $X_0,X_1,X_2,\ldots\subseteq X$ that cover $X$ and such that $I(X_j)\cong C(X_j)\otimes M_d$ for each $j\in\NN$.

Given $j$, let $\pi_j\colon I\to C(X_j)\otimes M_d$ be the corresponding quotient map, and set
\[
\mathcal{S}_j := 
\big\{
B\in\SubSep(I) : \pi_j(B) \text{ $d$-homogeneous}, \locdim(\Prim(\pi_j(B)))\leq l \big\}
\]
Applying Lemmas~\ref{prp:subcollections}(2) and~\ref{prp:LShomUntwisted}, we obtain that $\mathcal{S}_j$ is $\sigma$-complete and cofinal.
It follows that $\mathcal{S}:=\bigcap_{j=0}^\infty\mathcal{S}_j$ is $\sigma$-complete and cofinal as well.
Choose $B\in\mathcal{S}$ satisfying $B_0\subseteq B$.

To verify that $B$ is $d$-homogeneous, let $\varrho$ be an irreducible representation of $B$.
Since $B$ is $d$-subhomogeneous, we have $\dim(\varrho)\leq d$.
Extend $\varrho$ to an irreducible representation $\varrho'$ of $I$ (a-priori on a possibly larger Hilbert space).
Then there exists $x\in X$ such that $\varrho'$ is isomorphic to the quotient map to the fiber at $x$.
Let $j\in\NN$ such that $x\in X_j$.
Since $B$ belongs to $\mathcal{S}_j$, it exhausts the fiber at $x$, and we deduce that $\dim(\varrho)\geq d$.

To see that $\locdim(\Prim(B))\leq l$, let $K\subseteq\Prim(B)$ be a compact subset.
For each $j$, let $F_j\subseteq\Prim(B)$ be the closed subset corresponding to the quotient $\pi_j(B)$ of $B$.
Since $B$ belongs to $\mathcal{S}_j$, we have $\locdim(F_j)\leq l$.
Hence, $\dim(K\cap F_j)\leq l$.
We have $K=\bigcup_j (K\cap F_j)$, and therefore
\[
\dim(K) = \sup_j \dim(K\cap F_j) \leq l
\]
by the Countable Sum Theorem;
see \cite[Theorem~3.2.5, p.125]{Pea75DimThy}
(see also the introduction to \autoref{sec:subhom}).
\end{proof}

\section{Homogeneous \texorpdfstring{$C^*$-algebras}{C*-algebras}}
\label{sec:hom}

In this section, we compute the generator rank of homogeneous \ca{s};
see \autoref{prp:gr_hom}.
We first consider the unital, separable case (\autoref{prp:gr_hom_sep}), we then generalize to the unital, nonseparable case (\autoref{prp:gr_hom_unital}) and finally to the general case.
Unlike for commutative \ca{s}, the unital, separable case is highly nontrivial and it requires a delicate analysis of the codimension of certain submanifolds of $(M_d)^{n+1}_\sa$ (\autoref{prp:dimEdm}) in connection with a suitable version of the homotopy extension lifting property (\autoref{prp:HELP}).

The next result characterizes generating tuples in separable $C(X)$-algebras with simple fibers, and thus in particular in unital, separable, homogeneous \ca{s}.
Given a map $\varphi\colon D\to E$ between \ca{s} and $\vect{a}=(a_0,\ldots,a_n)\in D^{n+1}$, we set
\[
\varphi(\vect{a}) := (\varphi(a_0),\ldots,\varphi(a_n)) \in E^{n+1}.
\]

\begin{prp}
\label{prp:generators_C(X)-alg}
Let $X$ be a compact, metric space, and let $A$ be a separable $C(X)$-algebra such that all fibers are simple.
Let $n\in\NN$ and $\vect{a}\in A^{n+1}_\sa$.
Then $\vect{a}\in\Gen_{n+1}(A)_\sa$ if and only if the following are satisfied:
\begin{enumerate}
\item[(a)]
$\vect{a}$ generates each fiber, that is, $\vect{a}(x)\in\Gen_{n+1}(A(x))_\sa$ for each $x\in X$;
\item[(b)]
$\vect{a}$ separates the points of $X$ in the sense that for distinct $x,y\in X$, there is no isomorphism $\alpha\colon A(x)\to A(y)$ satisfying $\alpha(\vect{a}(x))=\vect{a}(y)$. 
\end{enumerate}
\end{prp}
\begin{proof}
Let us first assume that $\vect{a}\in\Gen_{n+1}(A)_\sa$.
For $x\in X$, let $\pi_x\colon A\to A(x)$ be the quotient map onto the fiber at $x$.
Since $\pi_x$ is a surjective ${}^*$-homomorphism, it maps $\Gen_{n+1}(A)_\sa$ to $\Gen_{n+1}(A(x))_\sa$, which verifies~(a).
Similarly, for distinct points $x,y\in X$, the map $\pi_x\oplus\pi_y\colon A\to A(x)\oplus A(y)$ is a surjective ${}^*$-homomorphism. 
It follows that $(\vect{a}(x),\vect{a}(y))=(\pi_x\oplus\pi_y)(\vect{a})\in\Gen_{n+1}(A(x)\oplus A(y))_\sa$.
To verify~(b), assume that $\alpha\colon A(x)\to A(y)$ is an isomorphism satisfying $\alpha(\vect{a}(x))=\vect{a}(y)$.
Then
\[
C^*((\vect{a}(x),\vect{a}(y))) = \big\{ (d,\alpha(d))\in A(x)\oplus A(y) : d\in A(x) \big\} \neq A(x)\oplus A(y),
\]
which contradicts that $(\vect{a}(x),\vect{a}(y))$ generates $A(x)\oplus A(y)$.
Thus, no such $\alpha$ exists.

Conversely, let us assume that~(a) and~(b) are satisfied.
Set $B:=C^*(\vect{a})$.
We need to prove $B=A$.
This follows from \cite[Lemma~3.2]{ThiWin14GenZStableCa} once we show that $B$ exhausts the fiber $A(x)$ for each $x\in X$, and that for distinct $x,y\in X$ there exists $b\in B$ such that $b(x)$ is full in $A(x)$ and $b(y)=0$.
The exhaustion of fibers follows directly from~(a).

Let $x,y\in X$ be distinct, and set $C:=(\pi_x\oplus\pi_y)(B)\subseteq A(x)\oplus A(y)$.
Note that $C$ is the sub-\ca{} of $A(x)\oplus A(y)$ generated by $(\vect{a}(x),\vect{a}(y))$.
If $C\neq A(x)\oplus A(y)$, using that $A(x)$ and $A(y)$ are simple, it follows from \cite[Lemma~5.10]{Thi21GenRnk} that there exists an isomorphism $\alpha\colon A(x)\to A(y)$ such that
\[
C = \big\{ (d,\alpha(d))\in A(x)\oplus A(y) : d\in A(x) \big\},
\]
which implies that $\alpha(\vect{a}(x))=\vect{a}(y)$.
Since this contradicts~(b), we deduce that $C=A(x)\oplus A(y)$.
Hence, there exists $b\in B$ such that $b(x)$ is full in $A(x)$ and $b(y)=0$.
\end{proof}

\begin{ntn}
\label{ntn:E}
For $d\geq 2$ and $n\in\NN$, we set:
\begin{align*}
E_d^{n+1}
:=(M_d)^{n+1}_\sa, \quad\text{ and }\quad
G_d^{n+1}
:=\Gen_{n+1}(M_d)_\sa \subseteq E_d^{n+1}.
\end{align*}

Note that $E_d^{n+1}$ is isomorphic to $\mathbb{R}^{(n+1)d^2}$ as topological vector spaces.
In particular, $E_d^{n+1}$ is a (real) manifold with $\dim(E_d^{n+1})=(n+1)d^2$.

We let $\calU_d$ denote the unitary group of $M_d$.
It is a compact Lie group of dimension~$d^2$.
Every automorphism of $M_d$ is inner, and the kernel of $\calU_d\to\Aut(M_d)$ is the group of central unitaries $\mathbb{T} 1\subseteq \calU_d$.
Hence, $\Aut(M_d)$ is naturally isomorphic to $\PU_d:=\calU_d/(\mathbb{T} 1)$, the projective unitary group, which is a compact Lie group of dimension $d^2-1$.
Given $u\in\calU_d$, we use $[u]$ to denote its class in $\PU_d$.

The action $\PU_d\curvearrowright M_d$ induces an action $\PU_d\curvearrowright E_d^{n+1}$ by setting
\[
[u].\vect{a} := (ua_0u^*,\ldots,ua_nu^*)
\]
for $u\in\calU_d$ and $\vect{a}=(a_0,\ldots,a_n)\in E_d^{n+1}$.
\end{ntn}

\begin{pgr}
\label{pgr:homApproach}
Let $A$ be a unital, separable $d$-homogeneous \ca{}, and let $n\in\NN$.
Set $X:=\Prim(A)$.
We consider $A$ with its canonical $C(X)$-algebra structure, with each fiber isomorphic to $M_d$;
see \autoref{pgr:hom}.
Set
\[
\Gen^{\mathrm{fiber}}_{n+1}(A)_\sa
:= \big\{ \vect{a}\in A^{n+1}_\sa : \vect{a}(x)\in\Gen_{n+1}(A(x))_\sa \text{ for each } x\in X \big\}.
\]

Given $x\in X$, let $\pi_x\colon A\to A(x)$ denote the map to the fiber at $x$.
This induces a map $A^{n+1}_\sa\to (A(x))^{n+1}_\sa$, which we also denote by $\pi_x$.
Choose an isomorphism $A(x)\cong M_d$, which induces an isomorphism $(A(x))^{n+1}_\sa \cong E^{n+1}_d = (M_d)^{n+1}_\sa$.
Since the isomorphism $A(x)\cong M_d$ is unique up to an automorphism of $M_d$, we obtain a canonical homeomorphism $(A(x))^{n+1}_\sa/\Aut(A(x)) \cong E^{n+1}_d  / \PU_d$.
We let $\psi_x\colon A^{n+1}_\sa\to E^{n+1}_d/\PU_d$ be the resulting natural map.

Given $\vect{a}\in A^{n+1}_\sa$, one checks that $\psi_x(\vect{a})$ depends continuously on $x$.
This allows us to define $\Psi\colon A^{n+1}_\sa\to C(X,E^{n+1}_d/\PU_d)$ by 
\[
\Psi(\vect{a})(x) := \psi_x(\vect{a}),
\]
for $\vect{a}\in A^{n+1}_\sa$ and $x\in X$.
Restricting $\Psi$ to $\Gen^{\mathrm{fiber}}_{n+1}(A)_\sa$ gives a continuous map
\[
\Psi\colon \Gen^{\mathrm{fiber}}_{n+1}(A)_\sa \to C(X,G^{n+1}_d/\PU_d).
\]

We let $E(X,G^{n+1}_d/\PU_d)$ denote the set of continuous maps $X\to G^{n+1}_d/\PU_d$ that are injective.
By \autoref{prp:generators_C(X)-alg}, a tuple $\vect{a}\in A^{n+1}_\sa$ belongs to $\Gen_{n+1}(A)_\sa$ if and only if~(a): $\vect{a}\in\Gen^{\mathrm{fiber}}_{n+1}(A)_\sa$, and~(b): $\Psi(\vect{a})\in E(X,G^{n+1}_d/\PU_d)$.
Thus, to determine the generator rank of $A$, we need to answer the following questions:
\begin{enumerate}
\item[(a)]
When is $\Gen^{\mathrm{fiber}}_{n+1}(A)_\sa$ dense in $A^{n+1}_\sa$?
\item[(b)]
When is $E(X,G_d^{n+1}/\PU_d)$ dense in $C(X,G_d^{n+1}/\PU_d)$?
\end{enumerate}

Analogous as for the computation of the generator rank for unital, separable, commutative \ca{s} in \cite[Section~5]{Thi21GenRnk}, the answer to question~(a) is determined by $\dim(X)$, and the answer to~(b) is determined by $\dim(X\times X)$.
However, while in the commutative case the dominating condition was~(b) involving $\dim(X\times X)$, we will see that for $d$-homogeneous \ca{s} with $d\geq 2$ the dominating condition is~(a) involving $\dim(X)$.

To study~(a), we will determine the dimension of $E^{n+1}_d\setminus G^{n+1}_d$.
For this, we study the action $\PU_d\curvearrowright E_d^{n+1}$.
We will show that $G^{n+1}_d$ consists precisely of the tuples in $E_d^{n+1}$ with trivial stabilizer subgroup;
see \autoref{prp:G_trivial_stabilizer}.
This allows us to describe $E^{n+1}_d\setminus G^{n+1}_d$ as the union of the submanifolds corresponding to nontrivial stabilizer subgroups.
We then estimate the dimension of these submanifolds;
see \autoref{prp:dimEdm}.

To study~(b), we show that $G_d^{n+1}$ is an open subset of $E_d^{n+1}$;
see \autoref{prp:unitalGen}.
Hence, $G_d^{n+1}$ is a manifold with $\dim(G_d^{n+1})=\dim(E_d^{n+1})=(n+1)d^2$.
We let $G_d^{n+1}/\PU_d$ denote the quotient space.
Since $\PU_d$ is a compact Lie group of dimension $d^2-1$, it follows that $G_d^{n+1}/\PU_d$ is a manifold of dimension $(n+1)d^2-(d^2-1)=nd^2+1$.
We then use a result of \cite{Luu81ApproxByEmb} which characterizes when a continuous map to a manifold can be approximated by injective maps.

Finally, we use a version of the homotopy extension lifting property for the projection $G_d^{n+1}\to G_d^{n+1}/\PU_d$ (see \autoref{prp:HELP}) to show that a given tuple in $\Gen^{\mathrm{fiber}}_{n+1}(A)_\sa$ can be approximated by tuples that are mapped to $E(X,G_d^{n+1}/\PU_d)$ by $\Psi$.
\end{pgr}

\begin{pgr}
\label{pgr:orbitTypeDec}
Let $G$ be a compact Lie group, acting smoothly on a connected manifold $M$.
We briefly recall the orbit type decomposition.
For details, we refer the reader to \cite{Bre72CpctTransfGps} and \cite{Mei03GpActionMfld}.
We will later apply this for the action $\PU_d\curvearrowright E^{n+1}_d$.

The \emph{stabilizer subgroup} of $m\in M$ is
\[
\stab(m) := \big\{ g\in G : g.m=m \big\}.
\]
Two subgroups $H$ and $H'$ of $G$ are \emph{conjugate}, denoted $H\sim H'$, if there exists $g\in G$ such that $H=gH'g^{-1}$.
We let 
\[
T := \big\{ \{H': H'\sim \stab(m)\} : m\in M \big\}
\]
denote the collection of all conjugation classes of stabilizer subgroups.
Set
\[
M_t := \big\{ m\in M : \stab(m)\in t \big\}
\]
for $t\in T$.
We have $\stab(g.m)=g\stab(m)g^{-1}$ for all $g\in G$ and $m\in M$, which implies that each $M_t$ is $G$-invariant. 

Let us additionally assume that each $M_t$ is connected.
Then, by \cite[Theorem~1.30]{Mei03GpActionMfld}, each $M_t$ is a smooth embedded submanifold of $M$, and $M$ decomposes as a disjoint union $M=\bigcup_{t\in T} M_t$.
(See also \cite[Theorem~IV.3.3, p.182]{Bre72CpctTransfGps}.)
Further, this decomposition satisfies the \emph{frontier condition}: 
For all $t',t\in T$, if $M_{t'}\cap\overline{M_t}\neq\emptyset$ then $M_{t'}\subseteq \overline{M_t}$.
This defines a partial order on $T$ by setting $t'\leq t$ if $M_{t'}\subseteq\overline{M_t}$.
The \emph{depth} of $t\in T$ is defined as $\depth(t)=0$ if $t$ is maximal, and otherwise
\[
\depth(t) := \sup\big\{ k\geq 1 : t<t_1<t_2<\ldots<t_k \text{ for some } t_1,\ldots,t_k\in T \big\}.
\]

In many cases, one knows that $T$ is finite and contains a largest element;
see Sections~IV.3 and~IV.10 in \cite{Bre72CpctTransfGps}.

Set $M_{\mathrm{free}}:=\{m\in M : \stab(m)=\{1\}\}$.
If $M_{\mathrm{free}}\neq\emptyset$, then the conjugacy class of the trivial subgroup is the largest element in $T$, and $M_{\mathrm{free}}$ is an open submanifold of $M$.
The restriction of the action to $M_{\mathrm{free}}$ is free.
\end{pgr}

\begin{prp}
\label{prp:avoidSubmanifolds}
Retain the situation from \autoref{pgr:orbitTypeDec}.
Assume that $M$ is metrizable with metric $d_M$, $T$ is finite and $M\neq M_{\mathrm{free}}\neq\emptyset$.
Let $X$ be a compact, Hausdorff space.
Then the following are equivalent:
\begin{enumerate}
\item
$C(X,M_{\mathrm{free}})\subseteq C(X,M)$ is dense with respect to the metric $d(f,g):=\sup_{x\in X}d_M(f(x),g(x))$, for $f,g\in C(X,M)$;
\item
$\dim(X)<\dim(M)-\dim(M\setminus M_{\mathrm{free}})$.
\end{enumerate}
\end{prp}
\begin{proof}
Note that $T$ contains exactly one element of depth zero, namely the conjugacy class of $\{1\}$.
Therefore
\[
M\setminus M_{\mathrm{free}} = \bigcup_{t\in T, \depth(t)\geq 1} M_t.
\]
and it follows that
\[
\dim(M\setminus M_{\mathrm{free}}) 
= \max\big\{ \dim(M_t) : \depth(t)\geq 1 \big\}.
\]

To show that~(1) implies~(2), assume that $\dim(X)\geq\dim(M)-\dim(M\setminus M_{\mathrm{free}})$.
Choose $t\in T$ of depth $\geq 1$ such that $\dim(X)\geq\dim(M)-\dim(M_t)$.
As noted in \cite[Proposition~1.6]{BegEva91RRMatrixValued}, it follows that $C(X,M\setminus M_t)\subseteq C(X,M)$ is not dense, which implies that~(1) fails.

Assuming~(2), let us prove~(1).
Let $f\in C(X,M)$ and $\varepsilon>0$.
The proof is similar to that of Theorem~1.3 in \cite{BegEva91RRMatrixValued}.
We inductively change $f$ to avoid each~$M_t$, but instead of proceeding by the (co)dimension of the submanifolds, we use their depths.

It follows from the frontier condition that for each $t\in T$, the set $\overline{M_t}\setminus M_t$ is contained in the union of submanifolds $M_s$ with $s\in T$ and $\depth(s)>\depth(t)$.
Let $t_1,\ldots,t_K$ be an enumeration of the elements in $T$ with depth $\geq 1$, such that $\depth(t_1)\geq\depth(t_2)\geq\ldots\geq\depth(t_K)$.
Note that $M_{t_1}$ is a closed submanifold (since $t_1$ has maximal depth and thus $\overline{M_{t_1}}\setminus M_{t_1}=\emptyset$), and for each $j\geq 2$ the set $\overline{M_{t_j}}\setminus M_{t_j}$ is contained in $M_{t_1}\cup\ldots M_{t_{j-1}}$.
Further, every $M_{t_j}$ is a submanifold of codimension $\geq\dim(X)+1$.

By \cite[Lemma~1.4]{BegEva91RRMatrixValued}, if $Y\subseteq M$ is submanifold of codimension $\geq\dim(X)+1$, if $\delta>0$, and if $g\in C(X, M)$ satisfies $g(X)\cap (\overline{Y}\setminus Y)=\emptyset$, then there exists $g'\in C(X,M)$ such that $d(g,g')\leq\delta$ and $g'(X)\cap\overline{Y}=\emptyset$.
Set $f_0:=f$.
We will inductively find $f_k\in C(X,M)$ such that for each $k=1,\ldots,K$ we have
\[
d(f_{k-1},f_k) < \frac{\varepsilon}{2^k}, \quad\text{ and }\quad f_k(X)\cap \overline{M_{t_j}}=\emptyset \text{ for } j=1,\ldots,k.
\]

First, using that the boundary of $M_{t_1}$ is empty, we can apply \cite[Lemma~1.4]{BegEva91RRMatrixValued} to obtain $f_1\in C(X,M)$ such that
\[
d(f_0,f_1) < \frac{\varepsilon}{2}, \quad\text{ and }\quad f_1(X)\cap \overline{M_{t_1}}=\emptyset.
\]

For $k\geq 2$, assuming that we have chosen $f_{k-1}$, let $\delta_k$ denote the (positive) distance between the compact set $f_{k-1}(X)$ and $\overline{M_{t_1}}\cup\ldots\cup\overline{M_{t_{k-1}}}$.
Applying \cite[Lemma~1.4]{BegEva91RRMatrixValued}, we obtain $f_k\in C(X,M)$ such that
\[
d(f_{k-1},f_k) < \min\left\{ \frac{\varepsilon}{2^k},\delta_k \right\}, \quad\text{ and }\quad f_k(X)\cap \overline{M_{t_k}}=\emptyset.
\]
By choice of $\delta_k$, it follows that $f_k(X)$ is disjoint from $\overline{M_{t_1}}\cup\ldots\cup\overline{M_{t_k}}$.

Finally, the element $f_K$ belongs to $C(X,M_{\mathrm{free}})$ and satisfies $d(f,f_K)<\varepsilon$.
\end{proof}

\begin{pgr}
\label{pgr:Sub1}
We let $\Sub_1(M_d)$ denote the collection of sub-\ca{s} of $M_d$ that contain the unit of $M_d$.
Given $\vect{a}\in E^{n+1}_d:=(M_d)^{n+1}_\sa$, we set $C^*_1(\vect{a}):=C^*(\vect{a},1)\in\Sub_1(M_d)$. 
We let $\PU_d$ act on $\Sub_1(M_d)$ by $[u].B:=uBu^*$ for $u\in\calU_d$ and $B\in\Sub_1(M_d)$.
Given $B_1,B_2\in\Sub_1(M_d)$, we write $B_1\sim B_2$ if $B_1$ and $B_2$ lie in the same orbit of this action, that is, if $B_1=uB_2u^*$ for some $u\in\calU_d$.

Given $\vect{a}\in E^{n+1}_d$, we have $C^*_1(\vect{a})=M_d$ if and only if $C^*(\vect{a})$, and thus
\begin{align*}
G^{n+1}_d 
&:= \Gen_{n+1}(M_d)_\sa
= \big\{ \vect{a}\in (M_d)^{n+1}_\sa : C^*(\vect{a})=M_d \big\} \\
&= \big\{ \vect{a}\in (M_d)^{n+1}_\sa : C^*_1(\vect{a})=M_d \big\}.
\end{align*}

Given a sub-\ca{} $B\subseteq M_d$, we let $B' := \{ c \in M_d : bc=cb \text{ for all } b\in B \}$ denote its commutant.
We always have $B'\in\Sub_1(M_d)$, and by the bicommutant theorem we have $B''=B$ for all $B\in\Sub_1(M_d)$.
\end{pgr}

\begin{lma}
\label{prp:G_trivial_stabilizer}
Let $\vect{a}\in E^{n+1}_d$.
Then
\[
\stab(\vect{a}) = \big\{ [u] : u\in\calU(C^*(\vect{a})') \big\}.
\]
Further, we have $\vect{a}\in G^{n+1}_d$ if and only if $\stab(\vect{a})=\{[1]\}$.
\end{lma}
\begin{proof}
Given $u\in\calU_d$, we have $[u].\vect{a}=\vect{a}$ if and only if $uxu^*=x$ for every $x\in C^*(\vect{a})$.
This implies the formula for $\stab(\vect{a})$.

If $\vect{a}\in G^{n+1}_d$, then $C^*(\vect{a})'=\mathbb{C} 1$, which implies that $\stab(\vect{a})$ is trivial.
Conversely, assuming that $\vect{a}\in E_d^{n+1}\setminus G_d^{n+1}$, let us verify that $\vect{a}$ has nontrivial stabilizer subgroup.
Since~$C^*(\vect{a})\neq M_d$, we also have $C^*_1(\vect{a})\neq M_d$.
Using the bicommutant theorem, we deduce that $C^*_1(\vect{a})'$ is strictly larger than the center of $M_d$.
Using that $C^*(\vect{a})'=C^*_1(\vect{a})'$, we obtain a noncentral unitary in $C^*(\vect{a})'$.
\end{proof}

\begin{lma}
\label{prp:stabSim}
Let $\vect{a},\vect{b}\in E^{n+1}_d$.
Then we have $\stab(\vect{a}) \sim \stab(\vect{b})$ if and only if $C^*_1(\vect{a}) \sim C^*_1(\vect{b})$.
\end{lma}
\begin{proof}
Let $B_1,B_2\in\Sub_1(M_d)$.
If $u\in\calU_d$ satisfies $uB_1u^*=B_2$, then one checks $uB_1'u^*=B_2'$.
Using also that $B_1$ and $B_2$ agree with their bicommutants, we obtain:
\[
B_1\sim B_2 \quad\Leftrightarrow\quad B_1'\sim B_2'.
\]
Using that $C^*_1(\vect{a}) = C^*(\vect{a})''$ and $C^*_1(\vect{a})' = C^*(\vect{a})'$, and similarly $C^*_1(\vect{b}) = C^*(\vect{b})''$ and $C^*_1(\vect{b})' = C^*(\vect{b})'$, we need to show:
\begin{align*}
\stab(\vect{a}) \sim \stab(\vect{b})
\quad\Leftrightarrow\quad
C^*(\vect{a})' \sim C^*(\vect{b})'.
\end{align*}

To prove the forward implication, we assume that $\stab(\vect{a}) \sim \stab(\vect{b})$.
Let $v\in\calU_d$ such that $[v]\stab(\vect{a})[v]^{-1}= \stab(\vect{b})$.
Given $u\in\calU(C^*(\vect{a})')$, it follows from \autoref{prp:G_trivial_stabilizer} that
\[
[vuv^*] \in \stab(\vect{b})
= \big\{ [w] : w\in\calU(C^*(\vect{b})') \big\}.
\]
Using that $\mathbb{T} 1\subseteq\calU(C^*(\vect{b})')$, we obtain $vuv^*\in \calU(C^*(\vect{b})')$.
Since $C^*(\vect{a})'$ is spanned by its unitary elements, we get $vC^*(\vect{a})'v^*\subseteq C^*(\vect{b})'$.
The reverse inclusion is shown analogously, whence  $vC^*(\vect{a})'v^*= C^*(\vect{b})'$, that is, $C^*(\vect{a})'\sim C^*(\vect{b})'$.

Conversely, if $C^*(\vect{a})' \sim C^*(\vect{b})'$, let $v\in\calU_d$ such that $uC^*(\vect{a})'v^*=C^*(\vect{b})'$.
Using \autoref{prp:G_trivial_stabilizer}, we get $[v]\stab(\vect{a})[v]^{-1}= \stab(\vect{b})$, that is, $\stab(\vect{a}) \sim \stab(\vect{b})$.
\end{proof}

\begin{lma}
\label{prp:unitalGen}
Let $B$ be a finite-dimensional \ca{} and $n\geq 1$.
Then the set $\{\vect{a}\in B^{n+1}_\sa : C^*_1(\vect{a})=B\}$ is a path-connected, dense, open subset of $B^{n+1}_\sa$.
\end{lma}
\begin{proof}
Set $G:=\{\vect{a}\in B^{n+1}_\sa : C^*_1(\vect{a})=B\}$.

\emph{Denseness:}
By \cite[Lemma~7.2]{Thi21GenRnk}, we have $\gr(B)\leq 1\leq n$.
Since $B$ is unital and separable, it follows from \autoref{prp:grSep} that $\Gen_{n+1}(B)_\sa\subseteq B^{n+1}_\sa$ is dense.
Using that $\Gen_{n+1}(B)_\sa\subseteq G$, we get that $G$ is also dense in $B^{n+1}_\sa$.

\emph{Openness:}
Let $\mathcal{D}$ denote the family of sub-\ca{s} $D\subseteq B$ such that $D+\mathbb{C} 1_B$ is a proper sub-\ca{} of $B$ (that is, $C^*_1(D)\neq B$).
Then
\[
G = B^{n+1}_\sa \setminus \bigcup_{D\in\mathcal{D}} D^{n+1}_\sa.
\]
Thus, we need to show that $\bigcup_{D\in\mathcal{D}} D^{n+1}_\sa$ is a closed subset of $B^{n+1}_\sa$.

We let $\calU(B)$ denote the unitary group of $B$.
It naturally acts on $\mathcal{D}$ by setting $u.D:=uDu^*$ for $u\in\calU(B)$ and $D\in\mathcal{D}$.
Since $B$ is finite-dimensional, two sub-\ca{s} $D_1,D_2\subseteq B$ are unitarily equivalent if and only if $D_1\cong D_2$ and the inclusions induce the same maps in ordered $K_0$-theory.
It follows that the action $\calU(B)\curvearrowright\mathcal{D}$ has only finitely many orbits, and we choose representatives $D_1,\ldots,D_m\in\mathcal{D}$.
Then $\mathcal{D}=\bigcup_{j=1}^m \bigcup_{u\in\calU{B}} uD_ju^*$.

For each $j$, since $D_j$ is a closed subset of $B$, it follows that $(D_j)^{n+1}_\sa$ is a closed subset of $B^{n+1}_\sa$.
Since $B$ is finite-dimensional, $\calU(B)$ is compact, and it follows that
\[
\bigcup_{D\in\mathcal{D}} D^{n+1}_\sa
= \bigcup_{j=1}^m \bigcup_{u\in\calU(B)} u(D_j)^{n+1}_\sa u^*
\]
is closed, as desired.

\emph{Path-connectedness:}
We only sketch the argument for the case $B=M_d$ for some $d\geq 2$.
Let $\vect{a}\in\Gen_{n+1}(M_d)_\sa$.
Using that the unitary group of $M_d$ is path-connected, and that $a_0$ is unitarily equivalent to a diagonal matrix, we find a path in $\Gen_{n+1}(M_d)_\sa$ from $\vect{a}$ to some $\vect{b}$ such that $b_0$ is diagonal.
By splitting multiple eigenvalues of $b_0$ and moving them away from zero, we find a path $(x_t)_{t\in[0,1]}$ inside the self-adjoint, diagonal matrices starting with $x_0=b_0$ and ending with some $x_1$ such that $x_1$ has $k$ distinct, nonzero diagonal entries, and such that $b_0\in C^*(x_t)$ for each $t\in[0,1]$.
Then $t\mapsto (x_t,b_1,\ldots,b_n)$ defines a path inside $\Gen_{n+1}(M_d)_\sa$.

Let $S$ denote the set of self-adjoint matrices in $M_d$ such that every off-diagonal entry is nonzero.
Note that $S$ is path-connected.
Next, we let $(y_t)_{t\in[0,1]}$ be a path inside the self-adjoint matrices starting with $y_0=b_1$, ending with some matrix $y_1$ that has the eigenvalues $1,2,\ldots,d$ and such that $y_t$ belongs to $S$ for every $t\in(0,1]$.
Note that $x_1$ and $y_t$ generated $M_d$ for every $t\in(0,1]$.
It follows that $t\mapsto (x_1,y_t,b_2,\ldots,b_n)$ defines a path inside $\Gen_{n+1}(M_d)_\sa$.

Conjugating by a suitable path of unitaries, we find a path in $\Gen_{n+1}(M_d)_\sa$ from $(x_1,y_1,b_2\ldots,b_n)$ to some $\vect{c}=(c_0,c_1,\ldots,c_n)$ such that $c_1=\diag(1,2,\ldots,d)$.
Arguing as above, we find a path in $\Gen_{n+1}(M_d)_\sa$ that changes $c_0$ to the matrix $\tilde{c_0}$ with all entries $1$.
Then $\tilde{c_0}$ and $c_1$ generate $M_d$.

Then $t\mapsto (\tilde{c_0},c_1,(1-t)c_2,\ldots,(1-t)c_n)$ is a path in $\Gen_{n+1}(M_d)_\sa$ connecting to $(\tilde{c_0},c_1,0\ldots,0)$.
Thus, every $\vect{a}\in\Gen_{n+1}(M_d)_\sa$ is path-connected to the same element.
\end{proof}

\begin{pgr}
\label{pgr:parametrization}
Let $d\geq 2$, and $n\in\NN$.
The compact Lie group $\PU_d$ acts smoothly on the manifold $E^{n+1}_d:=(M_d)^{n+1}_\sa$.
We will describe the corresponding orbit type decomposition of $E^{n+1}_d$.

Given $\vect{a},\vect{b}\in E^{n+1}_d$, by \autoref{prp:stabSim} we have $\stab(\vect{a})\sim\stab(\vect{b})$ if and only if $C^*_1(\vect{a})\sim C^*_1(\vect{b})$.
Moreover, given $B\in\Sub_1(M_d)$, there exists $\vect{a}\in E^{n+1}_d$ with $B=C^*_1(\vect{a})$.
It follows that the orbit types of $\PU_d\curvearrowright E^{n+1}_d$ naturally correspond to the orbit types of the action $\PU_d\curvearrowright\Sub_1(M_d)$.

Given $B_1,B_2\in\Sub_1(M_d)$, it is well-known that $B_1\sim B_2$ if and only if $B_1$ and~$B_2$ are isomorphic, that is, $B_1\cong B_2\cong \oplus_{j=1}^L M_{d_j}$ for some $L,d_1,\ldots,d_L\geq 1$, and if for each $j$ the maps $M_{d_j}\to B_1\to M_d$ and $M_{d_j}\to B_2\to M_d$ have the same mul\-ti\-plic\-i\-ty~$m_j$.
Thus, to parametrize the orbit types of $\PU_d\curvearrowright\Sub_1(M_d)$, we consider
\[
T_0 := \left\{ \big( (d_1,\ldots,d_L),(m_1,\ldots,m_L) \big) : L,d_j,m_j\geq 1,
\sum_{j=1}^L d_jm_j=d \right\}.
\]

Given $(\vect{d},\vect{m})\in T_0$, we let $B(\vect{d},\vect{m})\subseteq M_d$ be the sub-\ca{} of block diagonal matrices, with $m_1$ equal blocks of size $d_1$, followed by $m_2$ equal blocks of size $d_2$, and so on.
We point out that the numbers $d_1,\ldots,d_L$ are not required to be distinct.
For example $B((d),(1))=M_d$, $B((1),(d))=\mathbb{C} 1$, and $B((1,\ldots,1),(1,\ldots,1))$ is the algebra of diagonal matrices.

We define an equivalence relation on $T_0$ by setting $(\vect{d},\vect{m})\sim(\vect{d}',\vect{m}')$ if all tuples $\vect{d},\vect{m},\vect{d}',\vect{m}'$ contain the same number of elements, say $L\geq 1$, and if there is a permutation $\sigma$ of $\{1,\ldots,L\}$ such that
\[
d_j=d_{\sigma(j)}',\quad
m_j=m_{\sigma(j)}' \text{ for } j=1,\ldots,L.
\]
For example, we have $((2,2),(1,2))\sim((2,2),(2,1))$, but $((2,2),(1,2))\nsim((2),(3))$.

We have $(\vect{d},\vect{m})\sim(\vect{d}',\vect{m}')$ if and only if $B(\vect{d},\vect{m})\sim B(\vect{d}',\vect{m}')$.

Set $T:=T_0/_\sim$.
Given $(\vect{d},\vect{m})\in T_0$, we let $[\vect{d},\vect{m}]$ denote its equivalence class in~$T$.
For every $B\in\Sub_1(M_d)$ there exists $(\vect{d},\vect{m})\in T_0$ such that $B\sim B(\vect{d},\vect{m})$.
It follows that the orbit types of $\PU_d\curvearrowright\Sub_1(M_d)$ are parametrized by $T$:
\[
\Sub_1(M_d)/\PU_d 
\quad = \quad \Sub_1(M_d)/_\sim
\quad \cong \quad T_0/_\sim
\quad = \quad T.
\]

Given $[\vect{d},\vect{m}]\in T$, set
\[
E_{[\vect{d},\vect{m}]} := \big\{ \vect{a}\in E^{n+1}_d : C^*_1(\vect{a})\sim B(\vect{d},\vect{m}) \big\}.
\]
Then $E_{[\vect{d},\vect{m}]}$ is the submanifold of $E^{n+1}_d$ corresponding to orbit type $[\vect{d},\vect{m}]$, and the orbit type decomposition (as described in \autoref{pgr:orbitTypeDec}) for $\PU_d\curvearrowright E^{n+1}_d$ is:
\[
E^{n+1}_d = \bigcup_{[\vect{d},\vect{m}] \in T}  E_{[\vect{d},\vect{m}]}.
\]

By \autoref{prp:G_trivial_stabilizer}, a tuple $\vect{a}\in E^{n+1}_d$ has trivial stabilizer group if and only if $\vect{a}$ belongs to $G^{n+1}_d$.
It follows that $G^{n+1}_d = E_{[(d),(1)]}$, and in the notation of \autoref{pgr:orbitTypeDec}, with $M=E^{n+1}_d$, we have $M_{\mathrm{free}}=G^{n+1}_d$.
\end{pgr}

\begin{lma}
\label{prp:dimEdm}
Let $[\vect{d},\vect{m}]\in T$ with $[\vect{d},\vect{m}]\neq[(d),(1)]$.
Then $E_{[\vect{d},\vect{m}]}$ is a connected submanifold of $E^{n+1}_d$ satisfying	
\[
\dim( E_{[\vect{d},\vect{m}]} ) \leq (n+1)d^2 -2n(d-1).
\]
Further, $\dim( E_{[(d-1,1),(1,1)]} ) = (n+1)d^2 -2n(d-1)$.
\end{lma}
\begin{proof}
Set $B:=B(\vect{d},\vect{m})$.
Note that a tuple $\vect{a}\in E^{n+1}_d$ belongs to $E_{[\vect{d},\vect{m}]}$ if and only if $C^*_1(\vect{a})\sim B$.
Set
\[
F:= \big\{ \vect{a}\in E^{n+1}_d : C^*_1(\vect{a})=B \big\}.
\]
By \autoref{prp:unitalGen}, $F$ is connected.
Since every orbit in $E_{[\vect{d},\vect{m}]}$ meets $F$, and since $\PU_d$ is connected, it follows that $E_{[\vect{d},\vect{m}]}$ is connected as well.

By \cite[Theorem~IV.3.8]{Bre72CpctTransfGps}, if a compact Lie group $L$ acts smoothly on a connected manifold $M$ such that all orbits have the same type, then $\dim(M)=\dim(M/L)+\dim(L/K)$, where $K$ is the stabilizer subgroup of any element in $M$.

Let $K\subseteq\PU_d$ be the stabilizer subgroup of some element in $F$.
By considering the restricted action $\PU_d\curvearrowright E_{[\vect{d},\vect{m}]}$, we obtain that
\[
\dim( E_{[\vect{d},\vect{m}]} ) = \dim( E_{[\vect{d},\vect{m}]}/\PU_d ) + \dim( \PU_d/ K ).
\]

Closed subgroups of Lie groups are again Lie groups.
It follows that $K$ is a Lie group as well.
Since $K$ is acting freely on the connected manifold $\PU_d$ with only one orbit type, we also get $\dim( \PU_d/ K ) = \dim( \PU_d ) - \dim(K)$ and thus
\begin{align}
\label{eq:dimE-1}
\dim( E_{[\vect{d},\vect{m}]} )
= \dim( E_{[\vect{d},\vect{m}]}/\PU_d ) + \dim( \PU_d ) - \dim(K).
\end{align}

Set
\[
N:= \big\{ [u]\in\PU_d : uBu^*=B \big\},
\]
which is a closed subgroup of $\PU_d$.
Given $\vect{a}\in F$ and $[u]\in\PU_d$, we have $[u].\vect{a}\in F$ if and only if $[u]\in N$.
It follows that $N$ naturally acts on $F$.
Further, for each $\vect{a}\in F$, the $N$-orbit $N.\vect{a}$ agrees with $\PU_d.\vect{a}\cap F$.
Since every $\PU_d$-orbit in $E_{[\vect{d},\vect{m}]}$ meets $F$, we deduce that $E_{[\vect{d},\vect{m}]} / \PU_d \cong F / N$.
Note that $B^{n+1}_\sa$ is a linear space. 
By \autoref{prp:unitalGen}, $F$ is an open subset of $B^{n+1}_\sa$.
It follows that $F$ is a manifold satisfying
\[
\dim(F)
= \dim(B^{n+1}_\sa) 
= (n+1)\sum_{j=1}^L d_j^2.
\]

Analogous to \eqref{eq:dimE-1}, by considering the action of the compact Lie group $N$ on $F$, we obtain
\begin{align}
\label{eq:dimE-2}
\dim( F )
= \dim( F/N ) + \dim( N ) - \dim(K).
\end{align}

Note that $N$ contains $\{[u]:u\in\calU(B)\}$, which implies that
\[
\dim(N)\geq \left(\sum_{j=1}^L d_j^2 \right) - 1.
\]

Combining this estimate with \eqref{eq:dimE-1} and \eqref{eq:dimE-2}, using that $E_{[\vect{d},\vect{m}]} / \PU_d \cong F / N$, and that $[\vect{d},\vect{m}]\neq[(d),(1)]$, we get
\begin{align*}
\dim( E_{[\vect{d},\vect{m}]} )
&= \dim( F ) + \dim( \PU_d ) - \dim( N ) \\
&\leq \left( (n+1)\sum_{j=1}^L d_j^2 \right) 
+ \left( d^2 - 1 \right) 
- \left( \left(\sum_{j=1}^L d_j^2\right) -1 \right) \\
&= d^2 + n\sum_{j=1}^L d_j^2 \\
&\leq d^2 + n((d-1)^2+1)
= (n+1)d^2 -2n(d-1).
\end{align*}

For $[\vect{d},\vect{m}]=[(d-1,1),(1,1)]$ we have $B(\vect{d},\vect{m})\cong M_{d-1}\oplus\mathbb{C}\subseteq M_d$.
In this case, we get $N=\{[u]\in\PU_d: u\in\calU(M_{d-1}\oplus\mathbb{C})\}$ and thus $\dim(N)=(d-1)^2+1-1=(d-1)^2$.
It follows that
\begin{align*}
\dim( E_{[(d-1,1),(1,1)]} ) 
&= \dim( F ) + \dim( \PU_d ) - \dim( N ) \\
&= (n+1)((d-1)^2+1) + (d^2-1) - (d-1)^2 \\
&= (n+1)d^2 -2n(d-1). \qedhere
\end{align*}
\end{proof}

\begin{lma}
\label{prp:CXG-dense}
Let $X$ be a compact, Hausdorff space, $d\geq 2$, and $n\in\NN$.
Then the following are equivalent:
\begin{enumerate}[(1)  ]
\item
$C(X,G^{n+1}_d)\subseteq C(X,E^{n+1}_d)$ is dense;
\item
$\dim(X) < 2n(d-1)$.
\end{enumerate}
\end{lma}
\begin{proof}
We use the notation from \autoref{pgr:parametrization}.
The orbit type decomposition for the action $\PU_d\curvearrowright E^{n+1}_d$ is
\[
E^{n+1}_d = \bigcup_{[\vect{d},\vect{m}] \in T}  E_{[\vect{d},\vect{m}]}, \quad
E_{[\vect{d},\vect{m}]} := \big\{ \vect{a}\in E^{n+1}_d : C^*_1(\vect{a})\sim B(\vect{d},\vect{m}) \big\}.
\]
Further, $G^{n+1}_d = E_{[(d),(1)]}$, which is the submanifold of orbits with trivial stabilizers.
In the notation of \autoref{pgr:orbitTypeDec}, with $M=E^{n+1}_d$, we have $M_{\mathrm{free}}=G^{n+1}_d$.

Applying \autoref{prp:avoidSubmanifolds}, we obtain that $C(X,G^{n+1}_d)\subseteq C(X,E^{n+1}_d)$ is dense if and only if
\[
\dim(X) < \dim(E^{n+1}_d) - \dim(E^{n+1}_d\setminus G^{n+1}_d).
\]
Since $E^{n+1}_d\setminus G^{n+1}_d$ is the finite union of $E_{[\vect{d},\vect{m}]}$ for $[\vect{d},\vect{m}]\neq[(d),(1)]$, we obtain from \autoref{prp:dimEdm}
\[
\dim(E^{n+1}_d\setminus G^{n+1}_d)
= \max_{[\vect{d},\vect{m}]\neq[(d),(1)]} \dim( E_{[\vect{d},\vect{m}]} )
= (n+1)d^2 -2n(d-1).
\]
Now the result follows using that $\dim(E^{n+1}_d)=(n+1)d^2$.
\end{proof}

The next result provides the answer to question~(a) from \autoref{pgr:homApproach}.
Recall that $\Gen^{\mathrm{fiber}}_{n+1}(A)_\sa$ denotes the set of tuples that are fiberwise generators.

\begin{lma}
\label{prp:hom_approx_fiberwise_gen}
Let $A$ be a unital, separable $d$-homogeneous \ca{}, $d\geq 2$, and $n\in\NN$.
Then $\Gen^{\mathrm{fiber}}_{n+1}(A)_\sa\subseteq A^{n+1}_\sa$ is open.
Further, the following are equivalent:
\begin{enumerate}
\item
$\Gen^{\mathrm{fiber}}_{n+1}(A)_\sa$ is dense in $A^{n+1}_\sa$;
\item
$\dim(\Prim(A)) < 2n(d-1)$.
\end{enumerate}
\end{lma}
\begin{proof}
Set $X:=\Prim(A)$. 
Since $X$ is compact, and since the $M_d$-bundle associated to $A$ is locally trivial, we can choose closed subsets $X_1,\ldots,X_m\subseteq X$ that cover $X$ and such that $A(X_j)\cong C(X_j,M_d)$ for each $j$.
Let $\pi_j\colon A\to C(X_j,M_d)$ be the corresponding quotient map, which induces a natural map $A^{n+1}_\sa\to C(X_j,M_d)^{n+1}_\sa \cong C(X_j,E^{n+1}_d)$ that we also denote by $\pi_j$.

A tuple $\vect{a}\in A^{n+1}_\sa$ belongs to $\Gen^{\mathrm{fiber}}_{n+1}(A)_\sa$ if and only if $\pi_j(\vect{a})$ belongs to $C(X_j,G^{n+1}_d)$ for each $j$.
It follows from \autoref{prp:unitalGen} that $G^{n+1}_d\subseteq E^{n+1}_d$ is open.
Since $X$ is compact, we obtain that $C(X_j,G^{n+1}_d)\subseteq C(X_j,E^{n+1}_d)$ is always open.
Hence, $\Gen^{\mathrm{fiber}}_{n+1}(A)_\sa\subseteq A^{n+1}_\sa$ is open.

Since the intersection of finitely many open dense sets is again dense, we see that~(1) holds if and only if $C(X_j,G^{n+1}_d)\subseteq C(X_j,E^{n+1}_d)$ is dense for each~$j$.
By \autoref{prp:CXG-dense}, this is in turn equivalent to $\dim(X_j)<2n(d-1)$ for each $j$.
Using that $\dim(X)=\max_j\dim(X_j)$, this is finally equivalent to~(2).
\end{proof}

\begin{lma}
\label{prp:HELP}
Let $X$ be a compact, metric space, let $Y\subseteq X$ be closed, and let $F$ and $\widetilde{F}$ be continuous maps as in the diagram below such that $q\circ\widetilde{F}$ agrees with $F$ on $(Y\times[0,1])\cup (X\times\{0\})$.
\[
\xymatrix@R-5pt@C+20pt{
(Y\times[0,1])\cup (X\times\{0\})  \ar[r]^{\widetilde{F}} \ar@{^{(}->}[d] & G^{n+1}_d \ar[dd]^q \\
(Y\times[0,1])\cup (X\times [0,t])  \ar@{-->}[ur]_{\widetilde{H}} \ar@{^{(}->}[d]  \\
X\times[0,1] \ar[r]^{F} 
& G^{n+1}_d/\PU_d.
}
\]
Then there exist $t>0$ and a continuous map $\widetilde{H}$ making the above diagram commute.
\end{lma}
\begin{proof}
Using that the action $\PU_d\curvearrowright G^{n+1}_d$ is free, it follows that the quotient map $q\colon G^{n+1}_d\to G^{n+1}_d/\PU_d$ is the projection of a fiber bundle with base space $G^{n+1}_d/\PU_d$ and with fibers homeomorphic to $\PU_d$.
Using the homotopy lifting property for fiber bundles, we obtain $H\colon X\times[0,1]\to G^{n+1}_d$ such that
\[
q\circ H = F, \quad\text{ and }\quad
H(x,0) = \tilde{F}(x,0),\quad \text{for } x\in X.
\]
Next, we will correct $H$ to agree with $\widetilde{F}$ on $Y\times[0,t]$ for some $t>0$.

Given $(y,s)\in Y\times[0,1]$, we have
\[
q(H(y,s))
= F(y,s)
= q(\widetilde{F}(y,s)). 
\]
Let $c(y,s)\in\PU_d$ be the unique element such that $H(y,s)=c(y,s).\widetilde{F}(y,s)$.
This defines a map $c\colon Y\times[0,1]\to\PU_d$.
Using that the fiber bundle is locally trivial, we see that $c$ is continuous.
For every $y\in Y$, we have $H(y,0)=\widetilde{F}(y,0)$ and therefore $c(y,0)=1$.
We extend $c$ to a map $c\colon (Y\times[0,1])\cup (X\times\{0\})\to\PU_d$ by setting $c(x,0):=1$ for every $x\in X$.

Every Lie group is a (metrizable) locally contractible, finite-dimensional space and therefore an absolute neighborhood extensor;
see Theorems~4.2.33 and~1.2.7 in \cite{vMi01TopFunctionSp}.
This allows us to extend $c$ to a continuous map $\tilde{c}\colon U\to\PU_d$ defined on a neighborhood $U$ of $(Y\times[0,1])\cup (X\times\{0\})\subseteq X\times[0,1]$.
Then define $\widetilde{H}\colon U\to G^{n+1}_d$ by
\[
\widetilde{H}(x,s) := \tilde{c}(x,s).\widetilde{F}(x,s),\quad \text{for } (x,s)\in U\subseteq X\times[0,1].
\]

Choose $t>0$ such that $(Y\times[0,1])\cup (X\times [0,t])\subseteq U$.
Then the restriction of $\widetilde{H}$ to $(Y\times[0,1])\cup (X\times [0,t])$ has the desired properties.
\end{proof}

\begin{lma}
\label{prp:gr_hom_sep}
Let $A$ be a unital, separable $d$-homogeneous \ca{}, $d\geq 2$.
Then
\[
\gr(A) = \left\lceil\frac{\dim(\Prim(A))+1}{2d-2}\right\rceil.
\]
\end{lma}
\begin{proof}
Set $X:=\Prim(A)$.
Since $A$ is noncommuative, we have $\gr(A)\geq 1$ by \cite[Proposition~5.7]{Thi21GenRnk}.
We also have $\left\lceil\frac{\dim(X)+1}{2d-2}\right\rceil\geq 1$ for every value of $\dim(X)$.
Thus, it is enough to show that for every $n\geq 1$ the following holds:
\[
\gr(A)\leq n
\quad\Leftrightarrow\quad
\dim(X)<2n(d-1).
\]

Recall that we use $E(X,G^{n+1}_d/\PU_d)$ to denote the set of \emph{injective} continuous maps $X\to G^{n+1}_d/\PU_d$.
As explained in \autoref{pgr:homApproach}, we have the following inclusions and maps:
\[
\xymatrix@R-10pt{
\Gen_{n+1}(A)_\sa \ar@{}[r]|{\subseteq} \ar[d]^{\Psi} 
& \Gen^{\mathrm{fiber}}_{n+1}(A)_\sa \ar@{}[r]|{\subseteq} \ar[d]^{\Psi} & A^{n+1}_\sa \\
E(X,G^{n+1}_d/\PU_d) \ar@{}[r]|{\subseteq} & C(X,G^{n+1}_d/\PU_d).
}
\]

Assume that $\gr(A)\leq n$.
Since $A$ is separable, it follows from \autoref{prp:grSep} that $\Gen_{n+1}(A)_\sa\subseteq A^{n+1}_\sa$ is dense.
Since $\Gen_{n+1}(A)_\sa\subseteq \Gen^{\mathrm{fiber}}_{n+1}(A)_\sa$, we deduce from \autoref{prp:hom_approx_fiberwise_gen} that $\dim(X)<2n(d-1)$.

Conversely, assume that $\dim(X)<2n(d-1)$.
Applying \autoref{prp:hom_approx_fiberwise_gen}, we see that $\Gen^{\mathrm{fiber}}_{n+1}(A)_\sa\subseteq A^{n+1}_\sa$ is dense and open.
Further, by \autoref{prp:generators_C(X)-alg}, a tuple $\vect{a}\in \Gen^{\mathrm{fiber}}_{n+1}(A)_\sa$ belongs to $\Gen_{n+1}(A)_\sa$ if and only if $\Psi(\vect{a})$ belongs to $E(X,G^{n+1}_d/\PU_d)$.
Thus, to verify $\gr(A)\leq n$, it suffices to show the following:

\emph{
Let $\vect{a}\in\Gen^{\mathrm{fiber}}_{n+1}(A)_\sa$ and $\varepsilon>0$.
Then there exists $\vect{b}\in\Gen^{\mathrm{fiber}}_{n+1}(A)_\sa$ such that
\[
\vect{b}=_\varepsilon\vect{a}, \quad\text{ and }\quad 
\Psi(\vect{b})\in E(X,G^{n+1}_d/\PU_d).
\]
}

By \cite[Theorem~5.1]{Luu81ApproxByEmb}, if $M$ is a metrizable manifold with $2\dim(X)<\dim(M)$, then $E(X,M)\subseteq C(X,M)$ is dense with respect to the metric $d(f,g)=\sup\{d_M(f(x),g(x)):x\in X\}$, where $d_M$ is a metric inducing the topology on $M$.

By \autoref{prp:unitalGen}, $G^{n+1}_d$ is an open subset of $E^{n+1}_d$ and therefore is a manifold of dimension $(n+1)d^2$.
Further, $\PU_d$ is a compact Lie group of dimension $d^2-1$, acting freely on $G^{n+1}_d$.
Hence, as noted in the proof of \autoref{prp:dimEdm}, it follows from \cite[Theorem~IV.3.8]{Bre72CpctTransfGps} that $G^{n+1}_d/\PU_d$ is a manifold of dimension $(n+1)d^2-(d^2-1)=nd^2+1$.
By assumption, we have $\dim(X)<2n(d-1)$, and thus
\[
2\dim(X)
< 4n(d-1)
\leq nd^2+1.
\]
It follows that $E(X,G^{n+1}_d/\PU_d)$ is dense in $C(X,G^{n+1}_d/\PU_d)$.

Set $f:=\Psi(\vect{a})$.
Then $f\colon X\to G^{n+1}_d/\PU_d$ is a continuous map, which can be approximated arbitrarily closely by embeddings.
To complete the proof, we need to show that one of these embeddings is realized as $\Psi(\vect{b})$ for some $\vect{b}\in A^{n+1}_\sa$ close to~$\vect{a}$.
We will do this by successively appying our version of the homotopy extension lifting property proved in \autoref{prp:HELP}.

Every manifold is finite-dimensional and locally contractible and therefore an absolute neighborhood retract (ANR); see \cite[Theorem~4.2.33]{vMi01TopFunctionSp}.
Given a homotopy $H\colon X\times[0,1]\to M$ and $t\in[0,1]$, we let $H_t\colon X\to M$ be given by $H_t(x):=H(x,t)$.

\emph{Step~1:
We find a homotopy $F\colon X\times[0,1]\to G^{n+1}_d/\PU_d$ such that $F_0=f$ and such that $F_{1/k}$ belongs to $E(X,G^{n+1}_d/\PU_d)$ for every $k\geq 1$.
}

Set $M:=G^{n+1}_d/\PU_d$.
We use that $M$ is an ANR.
Given $\delta>0$, one says that $H\colon X\times[0,1]\to M$ is a \emph{$\delta$-homotopy} if $d(H_0,H_t)<\delta$ for all $t\in[0,1]$.
By \cite[Theorem~4.1.1]{vMi01TopFunctionSp}, for every $\delta>0$ there exists $\gamma>0$ such that for every $g\in C(X,M)$ satisfying $d(f,g)<\gamma$ there exists a $\delta$-homotopy $H\colon X\times[0,1]\to M$ with $H_0=f$ and $H_1=g$.
Given $n\in\NN$, we apply this for $\delta_n=\tfrac{1}{2^n}$, to obtain $\gamma_n>0$.
Using that $E(X,M)\subseteq C(X,M)$ is dense, choose $g_n\in E(X,M)$ satisfying $d(f,g_n)<\gamma_n$.
By choice of $\gamma_n$, we obtain a $\tfrac{1}{2^n}$-homotopy $H^{(n)}\colon X\times[0,1]\to M$ satisfying $H^{(n)}_0=f$ and $H^{(n)}_1=g_n$.

Next, we define $H\colon X\times[0,\infty)\to M$ by 
\[
H(x,t) = \begin{cases}
H^{(k)}(x,2k+1-t), &\text{ if } t\in[2k,2k+1] \\
H^{(k+1)}(x,t-2k-1), &\text{ if } t\in[2k+1,2k+2].
\end{cases}
\]

Thus, $H$ is the concatenation of the reverse of $H^{(1)}$, followed by $H^{(2)}$ and its reverse, and so on, as shown in the following picture:
\[
\makebox{
\begin{tikzpicture}
\draw [help lines, ->] (0,0) -- (5.2,0);
\draw [thick] (0,-0.1) -- (0,0.1);
\draw [thick] (1,-0.1) -- (1,0.1);
\draw [thick] (2,-0.1) -- (2,0.1);
\draw [thick] (3,-0.1) -- (3,0.1);
\draw [thick] (4,-0.1) -- (4,0.1);
\draw [thick] (5,-0.1) -- (5,0.1);
\node [below] at (0,0) {$0$};
\node [below] at (1,0) {$1$};
\node [below] at (2,0) {$2$};
\node [below] at (3,0) {$3$};
\node [below] at (4,0) {$4$};
\node [below] at (5,0) {$5$};
\node [above] at (0.5,0) {$H^{(0)}_{1-t}$};
\node [above] at (1.5,0) {$H^{(1)}_{t-1}$};
\node [above] at (2.5,0) {$H^{(1)}_{3-t}$};
\node [above] at (3.5,0) {$H^{(2)}_{t-3}$};
\node [above] at (4.5,0) {$H^{(2)}_{5-t}$};
\end{tikzpicture}
}
\]

Note that $H(\freeVar,2k)=g_k$ for each $k\in\NN$, and $\lim_{t\to\infty}H(x,t)=f(x)$ for every $x\in X$.
Let $\varrho\colon(0,1]\to[0,\infty)$ be a strictly decreasing, continuous map satisfying $\varrho(\tfrac{1}{k})=2k-2$ for $k\geq 1$.
Then $F\colon X\times[0,1]$ defined by $F(x,0)=f(x)$ and $F(x,t)=H(x,\varrho(t))$ for $t\in(0,1]$ has the desired properties.

\emph{Step~2:}
Since $X$ is compact, and since the $M_d$-bundle associated to $A$ is locally trivial, we can choose closed subsets $X_1,\ldots,X_m\subseteq X$ that cover $X$ and such that $A(X_j)\cong C(X_j,M_d)$ for each $j$.
Let $\pi_j\colon A\to C(X_j,M_d)$ be the corresponding quotient map.
Abusing notation, we also use $\pi_j$ to denote the naturally induced map
\[
\pi_j\colon A^{n+1}_\sa\to C(X_j,M_d)^{n+1}_\sa \cong C(X_j,E^{n+1}_d).
\]
Given $j,k\in\{1,\ldots,m\}$, both $\pi_j$ and $\pi_k$ induce an isomorphism between $A(X_j\cap X_k)$ and $C(X_j\cap X_k,M_d)$.
Let $c_{k,j}\colon X_j\cap X_k\to\PU_d=\Aut(M_d)$ be the continuous map such that $c_{k,j}(x).\pi_j(e)(x)=\pi_k(e)(x)$ for every $e\in A$ and $x\in X_j\cap X_k$.
Then 
\begin{align}
c_{k,j}(x).\pi_j(\vect{e})(x)=\pi_k(\vect{e})(x)
\end{align}
for every $\vect{e}\in A^{n+1}_\sa$ and $x\in X_j\cap X_k$.

\emph{Step~3:}
We will successively choose $t_1\geq t_2\geq\ldots\geq t_m>0$ and continuous maps
\[
H^{(k)}\colon X_k\times [0,t_k] \to G^{n+1}_d
\]
such that 
\begin{align}
\label{eq:liftHtpy1}
H^{(k)}(\freeVar,0)=\pi_k(\vect{a}), \quad\text{ and }\quad
q\circ H^{(k)} = F|_{X_k\times[0,t_k]},
\end{align}
and such that for every $j\leq k$ and $(x,s)\in (X_j\cap X_k)\times[0,t_k]$ we have
\begin{align}
\label{eq:liftHtpy2}
c_{k,j}(x).H^{(j)}(x,s)=H^{(k)}(x,s).
\end{align}

We start by setting $t_1:=1$.
The map $\pi_1(\vect{a})\colon X_1\to G^{n+1}_d$ satisfies $q\circ\pi_1(\vect{a})=f|_{X_1}$.
Thus, $\pi_1(\vect{a})$ is a lift of $F|_{X_1\times\{0\}}$.
Using the homotopy lifting property for fiber bundles, we obtain $H^{(1)}\colon X_1\times[0,1]\to G^{n+1}_d$ such that
\[
H^{(1)}_0 = \pi_1(\vect{a}), \quad\text{ and }\quad q\circ H^{(1)} = F|_{X_1\times[0,1]}.
\]

Next, assume that we have chosen $t_1\geq\ldots\geq t_{k-1}$ and $H^{(j)}$ for $j=1,\ldots,k-1$.
Set $Y_k:= X_k\cap (X_1\cup\ldots\cup X_{k-1})$, which is a closed subset of $X_k$.
We define $\widetilde{F}^{(k)}\colon (Y_{k}\times[0,t_{k-1}])\cup(X_{k}\times\{0\})\to G^{n+1}_d$ by
\[
\widetilde{F}^{(k)}(x,t) := \begin{cases}
c_{k,j}(x).H^{(j)}(x,t), & \text{ if } x\in X_{k}\cap X_j, \text{for } j\leq k-1; \\
\pi_{k}(\vect{a})(x), & \text{ if } t=0.
\end{cases}
\]
It follows from \eqref{eq:liftHtpy2} that $\widetilde{F}^{(k)}$ is well-defined.
Further, using \eqref{eq:liftHtpy1}, we obtain that $q\circ\widetilde{F}^{(k)}$ and $F$ agree on $(Y_{k}\times[0,t_{k-1}])\cup(X_k\times\{0\})$.
Applying \autoref{prp:HELP}, we obtain $t_k\in(0,t_{k-1}]$ and $H^{(k)}$ making the following diagram commute: 
\[
\xymatrix@R-10pt@C+20pt{
(Y_k\times[0,t_{k-1}])\cup (X_k\times\{0\})  \ar[r]^-{\widetilde{F}^{(k)}} \ar@{^{(}->}[d] & G^{n+1}_d \ar[dd]^q \\
(Y_k\times[0,t_{k-1}])\cup (X_k\times [0,t_k])  \ar@{-->}[ur]_>>>>>>>>>>{H^{(k)}} \ar@{^{(}->}[d]  \\
X_k\times[0,t_{k-1}] \ar[r]^{F} 
& G^{n+1}_d/\PU_d.
}
\]
One checks that $H^{(k)}$ has the desired properties.

\emph{Step~4:}
Let $t\in[0,t_m]$.
For each $j\in\{1,\ldots,m\}$, the map $H^{(j)}_t\colon X_j\to G^{n+1}_d$ defines an element in $\vect{b}_t^{(j)}\in C(X_j,M_d)^{n+1}_\sa$.
Given $j\leq k$ in $\{1,\ldots,m\}$ and $x\in X_j\cap X_k$, it follows from \eqref{eq:liftHtpy2} that
\[
c_{k,j}(x).\vect{b}_t^{(j)}(x) = \vect{b}_t^{(k)}(x).
\]
Thus, $\vect{b}_t^{(1)},\ldots,\vect{b}_t^{(m)}$ can be patched to give $\vect{b}_t\in A^{n+1}_\sa$ such that $\vect{b}_t^{(j)}=\pi_j(\vect{b}_t)$ for each $j$.
One checks that each $\vect{b}_t^{(j)}$ depends continuously on $t$, which implies that the map $[0,t_m]\to A^{n+1}_\sa$, $t\mapsto \vect{b}_t$, is continuous.
By construction, we have $\vect{b}_0=\vect{a}$, and $\Psi(\vect{b}_t)=F_t$ for each $t\in[0,t_m]$.
Using that $\vect{a}$ belongs to $\Gen^{\mathrm{fiber}}_{n+1}(A)_\sa$, which is an open subset of $A^{n+1}_\sa$, we can choose $k\geq 1$ such that
\[
\vect{a}=_\varepsilon\vect{b}_{1/k} \in \Gen^{\mathrm{fiber}}_{n+1}(A)_\sa.
\]

We have $\Psi(\vect{b}_{1/k})=F_{1/k}$, which by construction of~$F$ (Step~1) belongs to $E(X,G^{n+1}_d/\PU_d)$.
It follows that $\vect{b}_{1/k}\in\Gen_{n+1}(A)_\sa$.
\end{proof}

\begin{prp}
\label{prp:gr_hom_unital}
Let $A$ be a unital $d$-homogeneous \ca{}, $d\geq 2$.
Then:
\begin{align*}
\gr(A)
&= \left\lceil\frac{\dim(\Prim(A))+1}{2d-2}\right\rceil.
\end{align*}
\end{prp}
\begin{proof}
Set $n:=\gr(A)$ and $l:=\dim(\Prim(A))$, and then set
\begin{align*}
\mathcal{S}_1 &:= \big\{ B\in\SubSep(A) : 1\in B, \gr(B)\leq n \big\}, \quad\text{ and }\quad \\
\mathcal{S}_2 &:= \big\{ B\in\SubSep(A) : B \text{ $d$-homogeneous}, \locdim(\Prim(B))\leq l \big\}.
\end{align*}
As noted in \autoref{pgr:LS}, since $\gr$ satisfies~(D5) and~(D6), it follows that $\mathcal{S}_1$ is $\sigma$-complete and cofinal.
By \autoref{prp:LShom}, $\mathcal{S}_2$ is $\sigma$-complete and cofinal.
Hence, $\mathcal{S}_1\cap\mathcal{S}_2$ is $\sigma$-complete and cofinal as well.

Let $B\in \mathcal{S}_1\cap\mathcal{S}_2$.
Then $B$ is a unital, separable, $d$-homogeneous \ca{}.
Hence, $\dim(\Prim(B))=\locdim(\Prim(B))$, and by \autoref{prp:gr_hom_sep} we have
\[
\gr(B)
= \left\lceil\frac{\dim(\Prim(B))+1}{2d-2}\right\rceil.
\]

Thus, each $B\in\mathcal{S}_1\cap\mathcal{S}_2$ satisfies $\gr(B)\leq \left\lceil\tfrac{l+1}{2d-2}\right\rceil$.
Since $A$ is approximated by the family $\mathcal{S}_1\cap\mathcal{S}_2$, we obtain $\gr(A)\leq \left\lceil\tfrac{l+1}{2d-2}\right\rceil$ by \autoref{prp:gr_approx}.

To show the converse inequality, set
\[
m := \max \left\{ m_0\in\NN : n \geq \left\lceil\frac{m_0+1}{2d-2}\right\rceil \right\}.
\]
Then each $B\in\mathcal{S}_1\cap\mathcal{S}_2$ satisfies $\topdim(B)=\dim(\Prim(B))\leq m$.
Arguing with the topological dimension as in the proof of \autoref{prp:LShomUntwisted}, we deduce that $\dim(\Prim(A))=\topdim(A)\leq m$, and thus $n\geq \left\lceil\frac{\dim(\Prim(A))+1}{2d-2}\right\rceil$, as desired.
\end{proof}

\begin{thm}
\label{prp:gr_hom}
Let $A$ be a $d$-homogeneous \ca{}.
Set $X:=\Prim(A)$.
If $d=1$, then $\grPre(A)=\gr(A) = \locdim(X\times X)$.
If $d\geq 2$, then:
\begin{align*}
\grPre(A)=\gr(A)
&= \left\lceil\frac{\locdim(X)+1}{2d-2}\right\rceil.
\end{align*}
\end{thm}
\begin{proof}
For $d=1$, this follows from \autoref{prp:gr_abln}.
So assume that $d\geq 2$.
By \autoref{prp:grPre_gr}, we have $\grPre(A) \leq \gr(A)$.
Let $K\subseteq X$ be a compact subset.
The corresponding quotient $A(K)$ is a unital, $d$-homogeneous \ca{} with $\Prim(A(K))\cong K$.
Using \autoref{prp:gr_hom_unital} at the first step, and using \autoref{prp:gr_idealQuotExt} at the last step, we get
\[
\left\lceil\frac{\dim(K)+1}{2d-2}\right\rceil
= \gr(A(K))
= \grPre(A(K))
\leq \grPre(A).
\]
Since this holds for every compact subset of $X$, we deduce that
\[
\left\lceil\frac{\locdim(X)+1}{2d-2}\right\rceil
\leq \grPre(A) \leq \gr(A).
\]

To verify that $\gr(A)\leq\left\lceil\tfrac{\locdim(X)+1}{2d-2}\right\rceil$, set $l:=\locdim(X)$, which we may assume to be finite.
By \autoref{prp:LShom}, the collection
\[
\mathcal{S} := \big\{ B\in\SubSep(A) : B \text{ $d$-homogeneous}, \locdim(\Prim(B))\leq l \big\}
\]
is $\sigma$-complete and cofinal.
Let $B\in\mathcal{S}$.
We view $B$ as a locally trivial $M_d$-bundle over $Y:=\Prim(B)$.
Since $B$ is separable, $Y$ is $\sigma$-compact and thus $\dim(Y)=\locdim(Y)\leq l<\infty$.
By \cite[Lemma~2.5]{Phi07RSH}, the $M_d$-bundle has finite type.
Applying \cite[Proposition~2.9]{Phi07RSH}, we obtain a locally trivial $M_d$-bundle over the Stone-\v{C}ech-compactification $\beta Y$ extending the bundle associated to $B$.
This means that there is a unital, $d$-homogeneous \ca{} $D$ with $\Prim(D)\cong\beta(Y)$ and such that $B$ is an ideal in $D$.
Since $Y$ is a normal space, we have $\dim(\beta Y)=\dim(Y)$ by \cite[Proposition~6.4.3, p.232]{Pea75DimThy}.
Using \autoref{prp:gr_idealQuotExt} at the first step, and \autoref{prp:gr_hom_unital} at the second step, we get
\[
\gr(B)
\leq\gr(D) 
= \left\lceil\frac{\dim(\beta Y)+1}{2d-2}\right\rceil
\leq \left\lceil\frac{l+1}{2d-2}\right\rceil.
\]
Since $A$ is approximated by $\mathcal{S}$, we obtain $\gr(A)\leq \left\lceil\tfrac{l+1}{2d-2}\right\rceil$ by \autoref{prp:gr_approx}.
\end{proof}

In \autoref{prp:gr_sum_subhom}, we will generalize the following result to compute the generator rank of direct sums of subhomogeneous \ca{s}.

\begin{lma}
\label{prp:gr_sum_d-hom}
Let $A$ and $B$ be $d$-homogeneous \ca{s}.
Then 
\[
\gr(A\oplus B)
= \max \big\{ \gr(A),\gr(B) \big\}.
\]
\end{lma}
\begin{proof}
For $d=1$, this follows from \cite[Proposition~5.9]{Thi21GenRnk}.
So assume that $d\geq 2$.
Set $X:=\Prim(A)$, and $Y:=\Prim(B)$.
Then $A\oplus B$ is $d$-homogeneous with $\Prim(A\oplus B)\cong X\sqcup Y$, the disjoint union of $X$ and $Y$. 
Applying \autoref{prp:gr_hom} at the first and last step, we obtain
\begin{align*}
\gr(A\oplus B)
&= \left\lceil\frac{\locdim(X\sqcup Y)+1}{2d-2}\right\rceil 
= \left\lceil\frac{\max\{\locdim(X),\locdim(Y)\}+1}{2d-2}\right\rceil \\
&= \max\left\{ \left\lceil\frac{\locdim(X)+1}{2d-2}\right\rceil, \left\lceil\frac{\locdim(Y)+1}{2d-2}\right\rceil \right\} \\
&=\max\big\{ \gr(A),\gr(B) \big\}. \qedhere
\end{align*}
\end{proof}

\begin{rmk}
Let $A$ be a unital $d$-homogeneous \ca{}. 
Set $X:=\Prim(A)$.
If $d=1$, then $A\cong C(X)$, and by \autoref{prp:gr_abln} the generator rank of $A$ is $\dim(X\times X)$.
The value of $\dim(X\times X)$ is either $2\dim(X)$ or $2\dim(X)-1$, and accordingly we say that $X$ is of \emph{basic type} or of \emph{exceptional type};
see \cite[Proposition~5.3]{Thi21GenRnk}.

If $d\geq 2$, then by \autoref{prp:gr_hom_unital}, the generator rank of $A$ only depends on $\dim(X)$ (and $d$), but not on $\dim(X\times X)$.
Thus, in this case, the generator rank of~$A$ does not depend on whether $X$ is of basic or exceptional type.
\end{rmk}

\begin{rmk}
Let $m\geq 1$ and $d\geq 2$, and set $A=C([0,1]^m,M_d)$.
Let $\mathrm{gen}(A)$ denote the minimal number of self-adjoint generators for $A$.
By \cite[Theorem~4]{Nag04SingleGenRnkCa}, \cite[Corollary~3.2]{BegEva91RRMatrixValued} (see also \cite[Theorem~V.3.2.6]{Bla06OpAlgs}) and \autoref{prp:gr_hom_unital}, we have:
\[
\mathrm{gen}(A)=\left\lceil\frac{m-1}{d^2}+1\right\rceil,
\quad\quad
\rr(A)=\left\lceil\frac{m}{2d-1}\right\rceil,
\quad\quad
\gr(A)=\left\lceil\frac{m+1}{2d-2}\right\rceil.
\]
\end{rmk}

\section{Subhomogeneous \texorpdfstring{$C^*$-algebras}{C*-algebras}}
\label{sec:subhom}

In this section, we compute the generator rank of subhomogeneous \ca{s};
see \autoref{prp:gr_subhom}.
Recall that a \ca{} is \emph{$d$-subhomogeneous} (for some $d\geq 1$) if all of its irreducible representations have dimension at most $d$;
and it is \emph{subhomogeneous} if it is $d$-subhomogeneous for some $d$;
see \cite[Definition~IV.1.4.1, p.330]{Bla06OpAlgs}.
It is known that a \ca{} is subhomogeneous if and only if it is a sub-\ca{} of a homogeneous \ca{};
equivalently it is a sub-\ca{} of $C(X,M_d)$ for some compact, Hausdorff space $X$ and some $d\geq 1$.

Inductive limits of subhomogeneous \ca{s} are called \emph{approximately subhomogeneous}, or ASH-algebras for short.
As an application we show that every nonzero, $\mathcal{Z}$-stable ASH-algebra has generator rank one;
see \autoref{prp:gr_ZstableASH}.

To compute the generator rank of a subhomogeneous \ca{}, we use that it is a successive extension by homogeneous \ca{s}.
Using the results from \autoref{sec:hom}, we compute the generator rank of the homogeneous parts.
The crucial extra ingredient is \autoref{prp:CST-gr}, which allows us to compute the generator rank of the extension by a homogeneous \ca.

Given a \ca{} $A$, we equip the primitive ideal space $\Prim(A)$ with the hull-kernel topology;
see \cite[Section~II.6.5, p.111ff]{Bla06OpAlgs} for details.
Given an ideal $I\subseteq A$, the set $\hull(I):=\{J\in\Prim(A) : I\subseteq J\}$ is a closed subset of $\Prim(A)$, and this defines a natural bijection between ideals of $A$ and closed subsets of $\Prim(A)$.

\begin{lma}
\label{prp:charCoverPrim}
Let $A$ be a unital \ca{}, and let $(I_k)_{k\in\NN}$ be a decreasing sequence of ideals.
Then the following are equivalent:
\begin{enumerate}
\item
$\bigcup_k \hull(I_k)=\Prim(A)$;
\item
for each $\varphi\in A^*$, we have $\lim_{k\to\infty}\|\varphi|_{I_k}\| = 0$.
\end{enumerate} 
\end{lma}
\begin{proof}
For each $k\in\NN$, let $z_k$ denote the support projection of $I_k$ in $A^{**}$, and let $\pi_k\colon A\to A/I_k$ denote the quotient map.

\emph{Claim: Let $\varphi\in A^*_+$ and $k\in\NN$. 
Then $\|\varphi|_{I_k}\|=\varphi^{**}(z_k)$.}
To prove the claim, let $(h_\alpha)_\alpha$ denote an increasing, positive, contractive approximate unit of $I_k$. 
Since $\varphi|_{I_k}$ is a positive functional on $I_k$, we have $\|\varphi|_{I_k}\| = \lim_\alpha \varphi(h_\alpha)$ by \cite[Proposition~II.6.2.5]{Bla06OpAlgs}.
Using also that $z_k$ is the weak*-limit of $(h_\alpha)_\alpha$ in $A^{**}$, we get
\[
\|\varphi|_{I_k}\| = \lim_\alpha \varphi(h_\alpha) = \varphi^{**}(z_k),
\]
which proves the claim.

Let $S(A)$ denote the set of states on $A$, which is a compact, convex subset of $A^*$, and let $P(A)$ denote the pure states on $A$, which agrees with the set of extreme points in $S(A)$.
Given $a\in(A^{**})_\sa$, we let $\widehat{a}\colon S(A)\to\mathbb{R}$ be given by
\[
\widehat{a}(\varphi) = \varphi^{**}(a),
\]
for $\varphi\in S(A)$.
Then $\widehat{a}$ is affine.
If $a\in A_\sa$, then $\widehat{a}$ is continuous.
Given $k\in\NN$, let $(h_\alpha)_\alpha$ be an increasing approximate unit of $I_k$.
Then $\widehat{z_k}$ is the pointwise supremum of the increasing net $(\widehat{h_\alpha})_\alpha$ of continuous functions, and therefore lower-semicontinuous.

To show that~(1) implies~(2), assume that $\bigcup_k \hull(I_k)=\Prim(A)$.
Let $\varphi\in P(A)$.
Since every pure state on $A$ factors through an irreducible representation, there exists $k$ such that $\varphi$ factors through $\pi_k$.
Let $\bar{\varphi}\in(A/I_k)^*$ such that $\varphi=\bar{\varphi}\circ\pi_k$.
We have $\pi_k^{**}(z_k)=0$, and therefore
\[
\widehat{z_k}(\varphi)
= \varphi^{**}(z_k)
= \bar{\varphi}^{**}(\pi_k^{**}(z_k))
= 0.
\]

Thus, $(\widehat{z_k})_k$ is a decreasing sequence of lower semicontinuous, affine functions with $\lim_{k\to\infty}\widehat{z_k}(\varphi)=0$ for each $\varphi\in P(A)$.
By \cite[Proposition~1.4.10, p.36]{Alf71CpctCvxSets}, we have $\lim_{k\to\infty}\widehat{z_k}(\varphi)=0$ for each $\varphi\in S(A)$.
Applying the claim, it follows that $\lim_{k\to\infty}\|\varphi|_{I_k}\|=0$ for every $\varphi\in S(A)$.
Now~(2) follows using that every functional in $A^*$ is a linear combination of four states, 
by \cite[Theorem~II.6.3.4, p.106]{Bla06OpAlgs}.

To show that~(2) implies~(1), assume that $\bigcup_k \hull(I_k)\neq\Prim(A)$.
We will show that~(2) does not hold.
Let $J\subseteq A$ be a primitive ideal with $J\notin \bigcup_k \hull(I_k)$, and let $\bar{\varphi}$ be a pure state on $A/J$.
Let $\pi\colon A\to A/J$ denote the quotient map.
Set $\varphi:=\bar{\varphi}\circ\pi$, which is a pure state on $A$.
Let $k\in\NN$.
In general, the restriction of a pure state to an ideal is either zero or again a pure state.
Since $J\notin\hull(I_k)$, we have $\varphi|_{I_k}\neq 0$, and thus $\|\varphi|_{I_k}\|=1$.
Thus, $\lim_{k\to\infty}\|\varphi|_{I_k}\|=1 \neq 0$.
\end{proof}

\begin{prp}
\label{prp:StoneWeier}
Let $A$ be a \ca{}, let $(I_k)_{k\in\NN}$ be a decreasing sequence of ideals such that $\bigcup_k \hull(I_k)=\Prim(A)$, and let $B\subseteq A$ be a sub-\ca.
Assume that $B/(B\cap I_k)=A/I_k$ for each $k$.
Then $B=A$.
\end{prp}
\begin{proof}
We first reduce to the unital case.
So assume that $A$ is nonunital, let $\widetilde{A}$ denote its minimal unitization, and let $\bar{B}$ denote the sub-\ca{} of $\widetilde{A}$ generated by $B$ and the unit of $\widetilde{A}$.
For each $k\in\NN$, we consider $I_k$ as an ideal in $\widetilde{A}$.
Let $\pi_k\colon A\to A/I_k$ and $\pi_k^+\colon\widetilde{A}\to \widetilde{A}/I_k$ denote the quotient maps.
Note that $\widetilde{A}/I_k$ is naturally isomorphic to $(A/I_k)^+$, the forced unitization of $A/I_k$.
By assumption, $\pi_k(B)=\pi_k(A)$.
It follows that $\pi_k^+(\bar{B})=\pi_k^+(\widetilde{A})$.
Further, $\Prim(\widetilde{A})$ is the union of the hulls of the $I_k$.
Then, assuming that the results holds in the unital case, we obtain $\bar{B}=\widetilde{A}$, which implies $B=A$.

Thus, we may assume from now on that $A$ is unital.
To reach a contradiction, assume that $B\neq A$.
Using Hahn-Banach, we choose $\varphi\in A^*$ with $\varphi|_{B}\equiv 0$ and $\|\varphi\|=1$.
Apply \autoref{prp:charCoverPrim} to obtain $k$ such that $\|\varphi|_{I_k}\|<\tfrac{1}{8}$. 
Since every functional is a linear combination of four states (\cite[Theorem~II.6.3.4, p.106]{Bla06OpAlgs}), we obtain $\psi_m\in (I_k)^*_+$ with $\varphi|_{I_k}=\sum_{m=0}^3 i^m \psi_m$, and we may also ensure that $\|\psi_m\|\leq\|\varphi|_{I_k}\|<\tfrac{1}{8}$.
Using \cite[Theorem~II.6.4.16, p.111]{Bla06OpAlgs}, we can extend each $\psi_m$ to a positive functional $\tilde{\psi}_m\in A^*_+$ with $\|\tilde{\psi}_m\|=\|\psi_m\|$.
Set $\omega:=\varphi-\sum_{m=0}^3 i^m \tilde{\psi}_m$.
Then $\omega\in A^*$ satisfies $\omega|_{I_k}\equiv 0$ and $\|\varphi-\omega\|<\tfrac{1}{2}$.

Let $\bar{\omega}\in (A/I_k)^*$ satisfy $\omega=\bar{\omega}\circ\pi_k$. 
Given $a\in A$, use that $A/I_k=\pi_k(B)$ to choose $b\in B$ with $\pi_k(b)=\pi_k(a)$ and $\|b\|=\|\pi_k(a)\|$; 
see \cite[Proposition~II.5.1.5]{Bla06OpAlgs}.
Then $\omega(a)=\omega(b)$, and thus
\[
|\omega(a)|
= |\omega(b)|
\leq |\omega(b)-\varphi(b)|+|\varphi(b)|
\leq \|\omega-\varphi\| \|b\|
\leq \tfrac{1}{2} \|a\|.
\]
Hence $\|\omega\|\leq\tfrac{1}{2}$, and so $1=\|\varphi\|\leq\|\varphi-\omega\|+\|\omega\|<1$, which is a contradiction.
\end{proof}

\begin{prp}
\label{prp:CST-gr}
Let $A$ be a separable \ca{}, and let $(I_k)_{k\in\NN}$ be a \emph{decreasing} sequence of ideals satisfying $\bigcup_k \hull(I_k)=\Prim(A)$.
Then
\[
\grPre(A) = \sup_k \grPre(A/I_k), \quad\text{ and }\quad
\gr(A) = \sup_k \gr(A/I_k).
\]
\end{prp}
\begin{proof}
\emph{Part~1: We verify the equality for $\grPre$.}
For each $k$, set $B_k:=A/I_k$ and let $\pi_k\colon A\to B_k$ denote the quotient map.
By \autoref{prp:gr_idealQuotExt}, we have $\grPre(A)\geq\grPre(B_k)$.
It thus remains to prove $\grPre(A) \leq \sup_k \grPre(B_k)$.
Set $n:=\sup_k \grPre(B_k)$, which we may assume to be finite.
For each $k$, set
\[
D_k := \big\{ (a_0,\ldots,a_n)\in A^{n+1}_\sa : (\pi_k(a_0),\ldots,\pi_k(a_n))\in\Gen_{n+1}(B_k)_\sa \big\}.
\]
Since $\grPre(B_k)\leq n$, and since $B_k$ is separable, $\Gen_{n+1}(B_k)_\sa$ is a dense $G_\delta$-subset of $(B_k)^{n+1}_\sa$ by \autoref{prp:grSep}.
We deduce that $D_k$ is a dense $G_\delta$-subset of $A^{n+1}_\sa$.
Then, by the Baire category theorem, $D:=\bigcap_k D_k$ is a dense subset of $A^{n+1}_\sa$.

Let us show that $D\subseteq\Gen_{n+1}(A)_\sa$, which will imply that $\grPre(A)\leq n$.
Let $\vect{a}\in D$, and set $B:=C^*(\vect{a})\subseteq A$.
By construction, we have $\pi_k(B)=A/I_k$ for each~$k$.
Applying \autoref{prp:StoneWeier}, we get $B=A$, and thus $\vect{a}\in\Gen_{k}(A)_\sa$.

\emph{Part~2: We verify the equality for $\gr$.}
If $A$ is unital, this follows from Part~1.
So assume that $A$ is nonunital.
We consider $I_k$ as an ideal in $\widetilde{A}$.
As in the proof of \autoref{prp:StoneWeier}, we see that $\widetilde{A}/I_k\cong(A/I_k)^+$, and that $\Prim(\widetilde{A})$ is the union of the hulls of the $I_k$.
By \cite[Lemma~6.1]{Thi21GenRnk}, we have $\gr(B)=\gr(B^+)$ for every \ca{} $B$.
Applying Part~1 at the second step, we get
\[
\gr(A)
= \grPre(\widetilde{A})
= \sup_k \grPre(\widetilde{A}/I_k)
= \sup_k \grPre((A/I_k)^+)
= \sup_k \gr(A/I_k). \qedhere
\]
\end{proof}

\begin{lma}
\label{prp:gr_sums}
Let $A$ and $B$ be separable \ca{s}.
Assume that no nonzero quotient of $A$ is isomorphic to a quotient of $B$.
Then
\[
\grPre(A\oplus B)
= \max \big\{ \grPre(A),\grPre(B) \big\}, \ \text{ and }\
\gr(A\oplus B)
= \max \big\{ \gr(A),\gr(B) \big\}.
\]
\end{lma}
\begin{proof}
The equality for $\grPre$ follows directly from \cite[Proposition~5.10]{Thi21GenRnk} by considering the ideal $I:=A$.
Applying \autoref{prp:grPre_gr} at the first and last step, and using the formula for $\grPre$ and that $\rr(A\oplus B)=\max\{\rr(A),\rr(B)\}$ at the second step, we get
\begin{align*}
\gr(A\oplus B)
&= \max \big\{ \grPre(A\oplus B), \rr(A\oplus B) \big\} \\
&= \max \big\{ \grPre(A), \grPre(B), \rr(A), \rr(B) \big\} 
= \max \big\{ \gr(A), \gr(B) \big\}. \qedhere
\end{align*}
\end{proof}

\begin{thm}
\label{prp:gr_subhom}
Let $A$ be a subhomogeneous \ca{}.
For each $d\geq 1$, set $X_d:=\Prim_d(A)$, the subset of the primitive ideal space of $A$ corresponding to $d$-dimensional irreducible representations.
Then:
\begin{align*}
\grPre(A)
= \gr(A)
&= \max\left\{ \locdim(X_1\times X_1), \max_{d\geq 2} \left\lceil\frac{\locdim(X_d)+1}{2d-2}\right\rceil \right\} 
\end{align*}
\end{thm}
\begin{proof}
By \autoref{prp:grPre_gr}, we have $\grPre(A)\leq\gr(A)$.
Given $d\geq 1$, let $A_d$ denote the ideal-quotient of $A$ corresponding to the locally closed set $\Prim_d(A)\subseteq\Prim(A)$.
Applying \autoref{prp:gr_idealQuotExt}, we obtain $\grPre(A_d) \leq\grPre(A)$.
Note that $A_d$ is $d$-homogeneous.
In particular, $A_1\cong C_0(X_1)$.
Using \autoref{prp:gr_abln}, we get
\[
\locdim(X_1\times X_1) 
= \grPre(A_1) 
\leq \grPre(A).
\]
For $d\geq 2$, applying \autoref{prp:gr_hom}, we get
\[
\left\lceil\frac{\locdim(X_d)+1}{2d-2}\right\rceil
= \grPre(A_d)
\leq \grPre(A).
\]

It remains to verify that
\begin{align}
\label{eq:subhom}
\gr(A) \leq \max\left\{ \locdim(X_1\times X_1), \sup_{d\geq 2} \left\lceil\frac{\locdim(X_d)+1}{2d-2}\right\rceil \right\}.
\end{align}

Recall that a \ca{} is $m$-subhomogeneous if each of its irreducible representations has dimension at most $m$.
We prove the inequality \eqref{eq:subhom} by induction over $m$.
Note that $1$-subhomogeneous \ca{s} are precisely commutative \ca{s}, in which case \eqref{eq:subhom}  follows from \autoref{prp:gr_abln}.

Let $m\geq 2$, assume that \eqref{eq:subhom} holds for $(m-1)$-subhomogeneous \ca{s}, and assume that $A$ is $m$-subhomogeneous.
Let $I$ be the ideal of $A$ corresponding to the open subset $X_m\subseteq\Prim(A)$.
Note that $A/I$ is $(m-1)$-subhomogeneous.
Set
\[
n := \max\left\{ \locdim(X_1\times X_1), \sup_{d=1,\ldots,m-1} \left\lceil\frac{\locdim(X_d)+1}{2d-2}\right\rceil \right\}.
\]
By assumption of the induction, we have $\gr(A/I)\leq n$.
We need to prove that
\[
\gr(A)\leq\max\left\{ n, \left\lceil\frac{\locdim(X_m)+1}{2m-2}\right\rceil \right\}.
\]

Set $l := \locdim(X_m)$ and
\begin{align*}
\mathcal{S}_1 &:= \big\{ B\in\SubSep(A) : \gr(B/(B\cap I))\leq n \big\}, \quad\text{ and }\quad \\
\mathcal{S}_2 &:= \big\{ B\in\SubSep(A) : B\cap I \text{ $m$-homogeneous}, \locdim(\Prim(B\cap I))\leq l \big\}.
\end{align*}

As noted in \autoref{pgr:LS}, the collection $\{D\in\SubSep(A/I) : \gr(D)\leq n\}$ is  $\sigma$-complete and cofinal.
Applying \autoref{prp:subcollections}(2), we obtain that $\mathcal{S}_1$ is $\sigma$-complete and cofinal.
Similarly, using \autoref{prp:LShom} and \autoref{prp:subcollections}(1), we see that $\mathcal{S}_2$ is $\sigma$-complete and cofinal.
Hence, $\mathcal{S}_1\cap\mathcal{S}_2$ is $\sigma$-complete and cofinal as well.
Using \autoref{prp:gr_approx}, and using that $A$ is approximated by $\mathcal{S}_1\cap\mathcal{S}_2$, it suffices to verify that every $B\in\mathcal{S}_1\cap\mathcal{S}_2$ satisfies
\[
\gr(B)
\leq \max \left\{ n, \left\lceil\frac{l+1}{2m-2}\right\rceil \right\}.
\]

Let $B\in\mathcal{S}_1\cap\mathcal{S}_2$.
Set $J:=B\cap I$.
By construction, $J$ is $m$-homogeneous with $\locdim(\Prim(J))\leq l$, and $B/J$ is $(m-1)$-subhomogeneous with $\gr(B/J)\leq n$.
Note that $J$ is the ideal of $B$ corresponding to $\Prim_m(B)$.
Since $B$ is separable, $\Prim_m(B)$ is $\sigma$-compact.
Choose an increasing sequence $(Y_k)_{k\in\NN}$ of compact subsets of $\Prim_m(B)$ such that $\Prim_m(B)=\bigcup_k Y_k$.

For each $k$, note that $Y_k\subseteq \Prim_m(B)$ is closed, and let $J_k$ be the ideal of $J$ corresponding to the open subset $\Prim_m(B)\setminus Y_k$.
Considering $J_k$ as an ideal of~$B$, we have $B/J_k\cong (J/J_k) \oplus B/J$.
Since $J/J_k$ is $m$-homogeneous, and $B/J$ is $(m-1)$-subhomogeneous, no nonzero quotient of $J/J_k$ is isomorphic to a quotient of $B/J$.
Applying \autoref{prp:gr_sums}, we obtain 
\[
\gr(B/J_k)
= \gr\big( (J/J_k) \oplus B/J \big) 
= \max\big\{ \gr(J/J_k), \gr(B/J) \big\}.
\]

Since $J/J_k$ is a quotient of $J$, and $J$ is $m$-homogeneous with $\locdim(\Prim(J))\leq l$, it follows from \autoref{prp:gr_idealQuotExt} and \autoref{prp:gr_hom} that
\[
\gr(J/J_k) 
\leq \gr(J) 
= \left\lceil\frac{\locdim(\Prim(J))+1}{2m-2}\right\rceil
\leq \left\lceil\frac{l+1}{2m-2}\right\rceil.
\]
Applying \autoref{prp:CST-gr} at the first step, we obtain
\begin{align*}
\gr(B) 
= \sup_k \gr(B/J_k)
&= \sup_k \max\big\{ \gr(J/J_k), \gr(B/J) \big\} 
\leq 
\max \left\{ n, \left\lceil\frac{l+1}{2m-2}\right\rceil \right\},
\end{align*}
as desired.
\end{proof}

\begin{rmk}
\label{rmk:grSubhom}
Let $A$ be a $m$-subhomogeneous \ca{}.
For $d=1,\ldots,m$, let $A_d$ be the ideal-quotient of $A$ corresponding to the locally closed subset $\Prim_d(A)$. 
Then it follows from \autoref{prp:gr_subhom} that
\[
\gr(A) = \max\big\{ \gr(A_1), \ldots, \gr(A_m) \big\}.
\]
Analogous formulas hold for the real and stable rank;
see \cite[Lemma~3.4]{Bro16HigherRRSR}.
\end{rmk}

\begin{cor}
\label{prp:gr_sum_subhom}
Let $A$ and $B$ be subhomogeneous \ca{s}.
Then:
\[
\gr(A \oplus B)
= \max\big\{ \gr(A), \gr(B) \big\}.
\]
\end{cor}
\begin{proof}
Let $m\geq 1$ be such that $A$ and $B$ are $m$-subhomogeneous.
Let $A_d$ be the ideal-quotient of $A$ corresponding to $\Prim_d(A)$, and analogous for $B_d$, for $d=1,\ldots,m$.
Then $A_d$ and $B_d$ are $d$-homogeneous, and $A_d\oplus B_d$ is naturally isomorphic to the ideal-quotient of $A\oplus B$ corresponding to $\Prim_d(A\oplus B)$.
Applying \autoref{prp:gr_subhom} (see also \autoref{rmk:grSubhom}) at the first and last step, and using \autoref{prp:gr_sum_d-hom} at the second step, we get
\begin{align*}
\gr(A \oplus B)
&= \max_{d=1,\ldots,m} \gr(A_d\oplus B_d)
= \max_{d=1,\ldots,m} \max\big\{ \gr(A_d),\gr(B_d) \big\} \\
&=  \max\left\{ \max_{d=1,\ldots,m}\gr(A_d),\max_{d=1,\ldots,m}\gr(B_d) \right\} 
= \max\big\{ \gr(A), \gr(B) \big\}.\qedhere
\end{align*}
\end{proof}

It is natural to expect that the generator rank of a direct sum of \ca{s} is the maximum of the generator ranks of the summands.
The next result shows that this is the case if one of the summands is subhomogeneous.
In general, however, this is unclear;
see \cite[Questions~2.12, 6.4]{Thi21GenRnk}.

\begin{prp}
\label{prp:gr_sumWithSubhom}
Let $A$ and $B$ be \ca{s} and assume that $B$ is subhomogeneous.
Then:
\[
\grPre(A\oplus B)=\max\{\grPre(A),\grPre(B)\}, \quad\text{ and }\quad
\gr(A\oplus B)=\max\{\gr(A),\gr(B)\}.
\]
\end{prp}
\begin{proof}
Let $m\geq 1$ such that $B$ is $m$-subhomogeneous.
The proof proceeds analogous to that of \cite[Proposition~5.12]{Thi21GenRnk} (which is the result for $m=1$) by considering the smallest ideal $I\subseteq A$ such that $A/I$ is $m$-subhomogeneous
(instead of the smallest $I$ such that $A/I$ is commutative), and by using \autoref{prp:gr_sum_subhom} instaed of \cite[Proposition~5.9]{Thi21GenRnk}.
\end{proof}

\begin{lma}
\label{prp:gr_ZstableSH}
Let $A$ be a nonzero, subhomogeneous \ca{}.
Then $\gr(A\otimes\mathcal{Z})=1$.
\end{lma}
\begin{proof}
Given a finite subset $F\subseteq A$, set $A_F:=C^*(F)\subseteq A$.
Then $A_F$ is a finitely generated, subhomogeneous \ca.
By \cite[Theorem~1.5]{NgWin06NoteSH}, there is $k\in\NN$ such that $\locdim(\Prim_d(A_F))\leq k$ for every $d\geq 1$.
For $p,q\in\NN$ let $Z_{p,q}$ denote the dimension-drop algebra
\[
Z_{p,q} = \big\{ f\colon [0,1] \to M_p\otimes M_q : f \text{ continuous}, f(0)\in 1\otimes M_q, f(1)\in M_p\otimes 1 \big\}.
\]
For $p$ and $q$ sufficiently large (for example $p,q\geq k+2$), it follows from \autoref{prp:gr_subhom} that $\gr(A_F\otimes Z_{p,q})\leq 1$.
Using that $\mathcal{Z}$ is an inductive limit of dimension-drop algebras $Z_{p_n,q_n}$ with $\lim_n p_n = \lim_n q_n =\infty$, we have $\gr(A_F\otimes\mathcal{Z})\leq 1$ by \autoref{prp:gr_approx}.
The family of sub-\ca{s} $A_F\otimes\mathcal{Z}\subseteq A\otimes\mathcal{Z}$, indexed over the finite subsets of $A$ ordered by inclusion, approximates $A\otimes\mathcal{Z}$, whence $\gr(A\otimes\mathcal{Z})\leq 1$ by \autoref{prp:gr_approx}.

By \cite[Proposition~5.7]{Thi21GenRnk}, every noncommutative \ca{} has generator rank at least one, and thus  $\gr(A\otimes\mathcal{Z})=1$.
\end{proof}

\begin{thm}
\label{prp:gr_ZstableASH}
Every nonzero, $\mathcal{Z}$-stable ASH-algebra has generator rank one.

If $A$ is a separable, $\mathcal{Z}$-stable ASH-algebra, then a generic element of $A$ is a generator.
\end{thm}
\begin{proof}
Let $A$ be a nonzero, $\mathcal{Z}$-stable ASH-algebra.
Let $(A_\lambda)_\lambda$ be an inductive system of subhomogeneous \ca{s} such that $A\cong\varinjlim_\lambda A_\lambda$.
Then
\[
A\cong A\otimes\mathcal{Z} \cong \varinjlim_\lambda A_\lambda\otimes\mathcal{Z}.
\]
By \autoref{prp:gr_ZstableSH}, we have $\gr(A_\lambda\otimes\mathcal{Z})\leq 1$ for each $\lambda$.
Using \autoref{prp:gr_approx}, we get $\gr(A)\leq 1$.
Since $A$ is noncommutative, we deduce that $\gr(A)=1$ by \cite[Proposition~5.7]{Thi21GenRnk}.

If $A$ is also separable, then the generators in $A$ form a dense $G_\delta$-subset;
see \autoref{rmk:gr1}.
\end{proof}

\begin{rmk}
\label{rmk:grZstable}
Let $A$ be a unital, separable $\mathcal{Z}$-stable \ca.
It was shown in \cite[Theorem~3.8]{ThiWin14GenZStableCa} that $A$ contains a generator.
If $A$ is also approximately subhomogeneous, then \autoref{prp:gr_ZstableASH} shows that generators are even dense in $A$. 
I~expect that every $\mathcal{Z}$-stable \ca{} has generator rank one.
However, in general we do not even know that every $\mathcal{Z}$-stable \ca{} has real rank at most one.
\end{rmk}

\providecommand{\etalchar}[1]{$^{#1}$}
\providecommand{\href}[2]{#2}

\end{document}